\documentclass[a4paper]{amsart}
\usepackage{verbatim}
\usepackage[dvips]{color}
\usepackage{amssymb}
\usepackage[active]{srcltx}
\usepackage{latexsym}
\usepackage[dvipsnames, svgnames]{xcolor}
\usepackage{amsmath}
\usepackage{breakcites}
\usepackage[colorlinks,citecolor=.,hypertexnames=false,urlcolor=urlc,breaklinks=true, pagebackref, breaklinks=true]{hyperref}
 \usepackage{float}
\usepackage{mathrsfs}
\usepackage{cite}
\usepackage{esint}
\usepackage{comment}

\allowdisplaybreaks

\newtheorem{theorem}{Theorem}[section]
\newtheorem{proposition}{Proposition}[section]
\newtheorem{lemma}{Lemma}[section]
\newtheorem{corollary}{Corollary}[section]

\theoremstyle{definition}
\newtheorem{remark}{Remark}[section]
\newtheorem{definition}{Definition}

\floatstyle{ruled} \restylefloat{figure} \restylefloat{table}

\numberwithin{equation}{section}

\usepackage{pgf}
\usepackage{color}

\newcommand{\beq}{\begin{equation}}
\newcommand{\bea}[1]{\begin{array}{#1} }
\newcommand{\eeq}{ \end{equation}}
\newcommand{\ea}{ \end{array}}
\newcommand{\ep}{\epsilon}

\newcommand{\ga}{\gamma}

\newcommand{\ds}{\displaystyle}

\newcommand{\ran}{\rangle}
\newcommand{\lan}{\langle}
\newcommand{\Ga}{\Gamma}

\newcommand{\La}{\Lambda}
\newcommand{\ar}{\partial}
\newcommand{\si}{\sigma}

\newcommand{\om}{\omega}
\newcommand{\Om}{\Omega}

\newcommand{\sem}{\setminus}

\newcommand{\De}{\Delta}

\def \d {{\delta}}

\def\Xint#1{\mathchoice
{\XXint\displaystyle\textstyle{#1}}%
{\XXint\textstyle\scriptstyle{#1}}%
{\XXint\scriptstyle\scriptscriptstyle{#1}}%
{\XXint\scriptscriptstyle%
\scriptscriptstyle{#1}}%
\!\int}
\def\XXint#1#2#3{{\setbox0=\hbox{$#1{#2#3}{%
\int}$ }
\vcenter{\hbox{$#2#3$ }}\kern-.6\wd0}}
\def\barint{\,\Xint -} 
\def\bariint{\barint_{} \kern-.4em \barint}
\def\bariiint{\bariint_{} \kern-.4em \barint}
\renewcommand{\iint}{\int_{}\kern-.34em \int} 
\renewcommand{\iiint}{\iint_{}\kern-.34em \int} 

\def\mean#1{\mathchoice%
          {\mathop{\kern 0.2em\vrule width 0.6em height 0.69678ex depth -0.58065ex
                  \kern -0.8em \intop}\nolimits_{\kern -0.4em#1}}%
          {\mathop{\kern 0.1em\vrule width 0.5em height 0.69678ex depth -0.60387ex
                  \kern -0.6em \intop}\nolimits_{#1}}%
          {\mathop{\kern 0.1em\vrule width 0.5em height 0.69678ex
              depth -0.60387ex
                  \kern -0.6em \intop}\nolimits_{#1}}%
          {\mathop{\kern 0.1em\vrule width 0.5em height 0.69678ex depth -0.60387ex
                  \kern -0.6em \intop}\nolimits_{#1}}}

\def\vintslides_#1{\mathchoice%
          {\mathop{\kern 0.1em\vrule width 0.5em height 0.697ex depth -0.581ex
                  \kern -0.6em \intop}\nolimits_{\kern -0.4em#1}}%
          {\mathop{\kern 0.1em\vrule width 0.3em height 0.697ex depth -0.604ex
                  \kern -0.4em \intop}\nolimits_{#1}}%
          {\mathop{\kern 0.1em\vrule width 0.3em height 0.697ex depth -0.604ex
                  \kern -0.4em \intop}\nolimits_{#1}}%
          {\mathop{\kern 0.1em\vrule width 0.3em height 0.697ex depth -0.604ex
                  \kern -0.4em \intop}\nolimits_{#1}}}

\newcommand{\aveint}[2]{\mathchoice%
          {\mathop{\kern 0.2em\vrule width 0.6em height 0.69678ex depth -0.58065ex
                  \kern -0.8em \intop}\nolimits_{\kern -0.45em#1}^{#2}}%
          {\mathop{\kern 0.1em\vrule width 0.5em height 0.69678ex depth -0.60387ex
                  \kern -0.6em \intop}\nolimits_{#1}^{#2}}%
          {\mathop{\kern 0.1em\vrule width 0.5em height 0.69678ex depth -0.60387ex
                  \kern -0.6em \intop}\nolimits_{#1}^{#2}}%
          {\mathop{\kern 0.1em\vrule width 0.5em height 0.69678ex depth -0.60387ex
                  \kern -0.6em \intop}\nolimits_{#1}^{#2}}}

\def\eqn#1$$#2$${\begin{equation}\label#1#2\end{equation}}
\def\charfn_#1{{\raise1.2pt\hbox{$\chi
_{\kern-1pt\lower3pt\hbox{{$\scriptstyle#1$}}}$}}}
\def\diam{\operatorname{diam}}
\def\qq1{q_*}
\def\q2{q_{**}}
\def\dist{\operatorname{dist}}

\def\ep{\varepsilon}

\newdimen\vintbar
\def\vint{-\kern-\vintbar\int}

\def\B{\mathcal B}

\def\G{\mathcal G}
\def\H{\mathcal H}

\def\P{\mathcal P}

\def\mc{\mathcal M}
\def\nt{N^*}

\def\S{\mathcal S}

\def\U{\mathcal U}

\def\0{\boldsymbol 0}

\newcommand{\ree}{\mathbb{R}^{n+1}}

\newcommand{\Rn}{\mathbb{R}^n}

\newcommand{\R}{\mathbb R}

\newcommand{\pc}{\mathcal{P}}

\newcommand{\eps}{\epsilon}

\newcommand{\Hpn}{\mathcal{H}^{n+1}_{\text{p}}}

\newtoks\by
\newtoks\paper
\newtoks\book
\newtoks\jour
\newtoks\yr
\newtoks\pages
\newtoks\vol
\newtoks\publ

\def\name[#1, #2]{#1 #2}
\def\ota{{\hbox{\bf ???}}}
\def\cLear{\by=\ota\paper=\ota\book=\ota\jour=\ota\yr=\ota
\pages=\ota\vol=\ota\publ=\ota}
\def\endpaper{\the\by, \textit{\the\paper},
{\the\jour} \textbf{\the\vol} (\the\yr), \the\pages.\cLear}
\def\endbook{\the\by, \textit{\the\book},
\the\publ, \the\yr.\cLear}
\def\endpap{\the\by, \textit{\the\paper}, \the\jour.\cLear}
\def\endproc{\the\by, \textit{\the\paper}, \the\book, \the\publ,
\the\yr, \the\pages.\cLear}

\newcommand{\qc}{\mathcal{Q}}

\newcommand{\pom}{\partial\Omega}

\newcommand{\Xb}{{\bf X}}
\newcommand{\Yb}{{\bf Y}}
\newcommand{\Zb}{{\bf Z}}

\renewcommand{\d}{\, \mathrm{d}}

\usepackage{graphicx}

\begin{document}
\title[On big pieces approximations of parabolic hypersurfaces]{On big pieces approximations of\\ parabolic hypersurfaces}

\author{S. Bortz}
\address{Department of Mathematics
\\
University of Alabama
\\
Tuscaloosa, AL, 35487, USA}
\email{sbortz@ua.edu}
\author{J. Hoffman}
\address{Department of Mathematics
\\
University of Missouri
\\
Columbia, MO 65211, USA}
\email{jlh82b@mail.missouri.edu}
\author{S. Hofmann}
\address{
Department of Mathematics
\\
University of Missouri
\\
Columbia, MO 65211, USA}
\email{hofmanns@missouri.edu}
\author{J.L. Luna-Garcia}
\address{Department of Mathematics
\\
University of Missouri
\\
Columbia, MO 65211, USA}
\email{jlwwc@mail.missouri.edu}
\author{K. Nystr\"om}
\address{Department of Mathematics, Uppsala University, S-751 06 Uppsala, Sweden}
\email{kaj.nystrom@math.uu.se}

\thanks{The authors J. H., S. H., and J.L. L-G.  were partially supported by NSF grants
DMS-1664047 and DMS-2000048. K.N was partially supported by grant  2017-03805 from the Swedish research council (VR)
\\
\indent 2010  Mathematics Subject Classification: Primary 28A75, 42B37; 35K05, 35K10, 42B25.
\\
\indent Keywords and phrases: parabolic Lipschitz graph, parabolic uniform rectifiability, big pieces, parabolic measure, caloric measure.}

\date{\today}

\maketitle

\begin{abstract}
\noindent Let $\Sigma$ be a closed subset of $\R^ {n+1}$ which is parabolic Ahlfors-David regular and assume that  $\Sigma$ satisfies a 2-sided corkscrew condition. Assume, in addition, that  $\Sigma$ is either time-forwards Ahlfors-David regular, time-backwards Ahlfors-David regular, or parabolic uniform rectifiable. We then first prove that
 $\Sigma$  satisfies a {\it weak  synchronized two cube condition}. Based on this we are able to revisit the argument in \cite{NS} and prove that $\Sigma$ contains
{\it uniform big pieces of Lip(1,1/2) graphs}.  When  $\Sigma$ is parabolic uniformly rectifiable the construction can be refined and in this case we prove that $\Sigma$  contains
{\it uniform big pieces of regular parabolic Lip(1,1/2)  graphs}. Similar results hold if  $\Omega\subset\mathbb R^{n+1}$ is a connected component of $\mathbb R^{n+1}\setminus\Sigma$ and in this context we also give a parabolic counterpart of the main result in \cite{AHMNT} by proving
that if $\Om$ is a one-sided parabolic chord arc domain, and if  $\Sigma$ is parabolic uniformly rectifiable, then $\Om$ is in fact a parabolic chord arc domain. Our results give a flexible parabolic version of the classical (elliptic) result
of G. David and D. Jerison concerning the existence of uniform big pieces of Lipschitz graphs for sets satisfying a two disc condition.
\medskip

\end{abstract}


    \setcounter{equation}{0} \setcounter{theorem}{0}
    \section{Introduction}
    An important result due to G. David and D. Jerison \cite{DJ} states that if
    $\Sigma\subset\mathbb R^{n+1}$ is a closed set which is Ahlfors-David
    regular with respect to the surface measure $\sigma =\H^n\lfloor \Sigma$,
(i.e., the restriction of $n$-dimensional Hausdorff measure to $\Sigma$),
and if $\Sigma$ satisfies what they call a two disc condition, then $\Sigma$ contains uniform big pieces of Lipschitz graphs, see \cite{DJ}. This result and its ramifications have had deep impact on the theory of elliptic boundary value problems and  on the analysis of and on  uniformly rectifiable sets. Indeed, if $\Omega$ is one component of $\mathbb R^n\setminus\Sigma$, and if, in addition, $\Omega$ is an NTA-domain in the sense of \cite{JK}, then the result of G. David and D. Jerison implies that the harmonic measure on $\partial\Omega$ belongs 
to the Muckenhoupt class $A_\infty$ defined with
respect to $\sigma$;
equivalently, that the Dirichlet problem for Laplace's equation is solvable in such domains,
    with $L^p$ boundary data.
Furthermore, the results of \cite{DJ}, combined with the monumental works of G. David and S. Semmes \cite{DS}, \cite{DS1}, have led to additional characterizations of uniform rectifiability:  see, e.g. \cite{HMM}, \cite{GMT}.

    In this paper we are interested in parabolic counterparts of the result of  G. David and D. Jerison. In general the theory of parabolic boundary value problems, and the analysis of and on parabolic uniformly rectifiable sets, is less developed compared to
the elliptic counterparts and there are essentially only two strains of main results  in the field: the results due to Hofmann, Lewis, Murray, Silver, see \cite{H1}, \cite{HL}, \cite{HL1}, \cite{LM}, \cite{LS} and the results due to Hofmann, Lewis, Nystr{\"o}m, see \cite{HLN1}, \cite{HLN2}.

To indicate the scope of the present paper, we 
give rough statements of three theorems to be proved;  more precise statements, as well as definitions of our terminology,
will be given in the sequel.

{\setlength{\leftmargini}{75pt}	
\begin{enumerate}
	\item[\underline{Theorem 1}:]  {\it If $\Sigma$ is parabolic ADR and satisfies a ``weak time synchronized two cube condition", then $\Sigma$ contains big pieces of Lip-$(1,1/2)$ graphs.}
\end{enumerate}
}
This ``weak time-synchronized two cube condition" is automatically satisfied in the presence of
two sided corkscrews and parabolic uniform rectifiability, see Theorem \ref{urtscorkscrews}.  In fact, when $\Sigma$ is parabolic uniformly rectifiable, we can transfer regularity from
the set $\Sigma$ to the approximating graph, which gives the additional subtle $t$-regularity required for boundedness of parabolic singular integrals and for parabolic potential theory.
{\setlength{\leftmargini}{75pt}	
\begin{enumerate}
	\item[\underline{Theorem 2}:] {\it If $\Sigma$ is parabolic uniformly rectifiable and satisfies the two-sided corkscrew condition, then $\Sigma$ contains big pieces of {\tt regular} Lip-$(1,1/2)$ graphs.  If {\tt in addition},
	$\Sigma$
	is time-symmetric ADR, and $\Sigma=\partial\Omega$
	is the boundary of an open set $\Omega\subset \ree$ satisfying an {\tt interior} corkscrew condition, then
	$\Sigma$ satisfies a uniform {\tt interior} big pieces of {\it regular} Lip-$(1,1/2)$ graphs condition.
	}
	\smallskip
	\item[\underline{Corollary 1}:] {\it Let $\Omega\subset \ree$ be an open set
	satisfying an {\tt interior} corkscrew condition.  If $\Sigma=\partial\Omega$
	is parabolic uniformly rectifiable, time-symmetric ADR,
	and satisfies the two-sided
	corkscrew condition, then caloric measure $\omega$ is absolutely continuous with respect
	to ``surface measure" $\sigma$ on $\Sigma$, the parabolic ``Poisson kernel"
	$d\omega/d\sigma$
	verifies a uniform scale invariant
	weak Reverse H\"older estimate, and the $L^p$ (initial)-Dirichlet problem for the heat equation is solvable
	in $\Omega$, for some $p<\infty$.}
	\smallskip
	\item[\underline{Theorem 3}:] {\it If $\Omega$ is a one-sided parabolic chord arc domain, whose boundary is parabolic uniformly rectifiable, then $\Omega$ is a (two-sided) chord arc domain.  Moreover, the caloric measure of $\Omega$ satisfies a
	(local) $A_\infty$ condition.}
\end{enumerate}
}

A few comments are in order concerning Corollary 1, and Theorem 3.
By the main result of \cite{GH} (and the maximum principle), in the setting of  Theorem 3, and of
the second part of Theorem 2,
we immediately deduce that caloric measure satisfies a local, scale-invariant weak-$A_\infty$ condition
with respect to the natural parabolic analogue of surface measure on $\Sigma$. In the setting of Theorem 3,
caloric measure is doubling (by a fairly routine extension of the results of \cite{FGS} essentially following \cite{HLN2}), and so in that case
the weak-$A_\infty$ condition immediately improves to (strong) $A_\infty$. Furthermore
(again see \cite{GH}),
the (weak) $A_\infty$ condition is equivalent to $L^p$ solvability of the Dirichlet problem, for some $p<\infty$.
Prior to this result, $L^p$ solvability for finite $p$ had not been
established even for parabolic Chord arc domains with parabolic uniformly rectifiable boundaries.

In \cite{H1}, \cite{HL}, \cite{LM}, \cite{LS}, the authors established the correct notion of (time-dependent) regular parabolic Lipschitz graphs from the point of view of parabolic singular integrals and parabolic measure. To expand a bit on this, recall that  $\psi:\mathbb R^{n-1}\times\mathbb R\to \mathbb R$ is called Lip(1,1/2)
(or ``parabolic Lipschitz", and we shall sometimes simply write $\psi\in PLip$)
with constant $b$, if
\begin{eqnarray}\label{1.1}
|\psi(x,t)-\psi(y,s)|\leq b(|x-y|+|t-s|^{1/2})\
\end{eqnarray}
whenever $(x,t)\in\mathbb R^{n}$, $(y,s)\in\mathbb R^{n}$.
An open set $\Omega\subset\mathbb R^{n+1}$ is said to be an
(unbounded) Lip(1,1/2) (or $PLip$)  graph
domain, with constant $b$, if
\begin{eqnarray}\label{1.1a}
\Omega=\Omega_\psi=\{(x,x_n,t)\in\mathbb R^{n-1}\times\mathbb R\times\mathbb R:x_n>\psi(x,t)\}
\end{eqnarray} for some Lip(1,1/2)  function $\psi$ having Lip(1,1/2)  constant bounded by $b$.
A function
$ \psi = \psi ( x, t ) : \mathbb R^{n-1}\times\mathbb R\to \mathbb R$ is called a {\em Regular Parabolic} Lip(1,1/2)
function (and we shall write $\psi \in RPLip$)
with parameters $b_1$ and $b_2$, if $\psi$ 
satisfies   \begin{eqnarray} \label{1.7}
(i)&&|\psi(x,t)-\psi(y,t)|\leq b_{1}|x-y|, x, y\in \mathbb R^{n-1}, t\in\mathbb R,\notag\\
(ii)&& D_{1/2}^t\psi\in BMO(\mathbb R^n), \ \ \|D_{1/2}^t\psi\|_*\leq b_2<\infty.
\end{eqnarray}
 It is well known, and essentially due to Strichartz \cite{Stz}
 (but see also \cite{HL},  \cite{H2}), that if $\psi\in RPLip$
with parameters $b_1$ and $b_2$, then $\psi$ is Lip(1,1/2)  with constant $b=b(b_1,b_2)$.
 Here   $ D_{1/2}^t \psi  (x, t) $ denotes
the $ 1/2 $ derivative in $ t $ of $ \psi ( x, \cdot ), x \in \mathbb{R}^{n-1}$ fixed,
and $BMO(\Rn)$ is the usual
parabolic BMO space consisting of all $f\in L^1_{loc}(\Rn)$ (modulo constants)
such that
\[\|f\|_*:=\sup_R\fint_R|f(x,t)-f_R|\,\d x\d t <\infty\,,\]
where $R$ denotes a parabolic cube in $\Rn$, having dimensions $r\times ...\times r\times r^2$ for some $r>0$,
and $f_R:= \fint_R f$.

This half derivative in time can be defined by way of the Fourier
transform (at least for compactly supported $\psi$), or by the formula
\begin{eqnarray} \label{1.8}
 D_{1/2}^t  \psi (x, t)  \equiv \hat c \int_{ \mathbb R }
\, \frac{ \psi ( x, s ) - \psi ( x, t ) }{ | s - t |^{3/2} } \, \d s
\end{eqnarray} for properly chosen $ \hat c$.

As noted above, every $RPLip$ function is, in particular, Lip(1,1/2), i.e. the $RPLip$ condition is stronger than
Lip(1,1/2).
In fact, it is strictly stronger: there are examples of
functions $\psi$ which are  Lip(1,1/2) but not $RPLip$, see \cite{LS},
\cite{KW}.

We call $\Omega\subset\mathbb R^{n+1}$ an (unbounded) regular parabolic Lip(1,1/2)  graph domain
(or simply an $RPLip$ graph domain),
with constants $(b_1,b_2)$, if \eqref{1.1a} holds for some regular parabolic Lip(1,1/2)  function $\psi$ having constants
  $(b_1,b_2)$. An important insight in \cite{KW}, \cite{H1}, \cite{HL}, \cite{LM}, \cite{LS}, is that  from the
  perspective of  parabolic singular integrals and parabolic measure,
  the Lip(1,1/2)  condition alone
  does not suffice; instead the problems have
  to be framed in the context of  {\em regular} parabolic  Lip(1,1/2)  graph domains and this induces additional complexity in the parabolic setting compared to the elliptic situation.

In \cite{HLN1}, \cite{HLN2} the third and fifth author, together with John Lewis,  introduced a  notion of parabolic uniformly rectifiable sets and proved the existence of big pieces of regular parabolic Lipschitz graphs under the additional assumption that $\Sigma$ is Reifenberg flat in the parabolic sense.
These results were the first of their kind in the context of parabolic problems and the studies \cite{HLN1}, \cite{HLN2} were motivated by the study of parabolic or caloric measures in rough domains. Still, up to very recently no systematic and correct studies of parabolic uniformly rectifiable sets have
appeared in the literature.  In the series including \cite{BHHLN1}, \cite{BHHLN2}, and the present paper, we attempt to rectify this by conducting a thorough and detailed study of these objects.
In particular, in \cite{BHHLN1} we prove,
among other
things, that parabolic uniformly rectifiable sets  satisfy a corona decomposition
with respect to {\em regular} Lip(1,1/2) graphs. In \cite{BHHLN2}, we obtain a converse to this  result from  \cite{BHHLN1},
as we prove that corona decomposition
with respect to {\em regular} Lip(1,1/2) graphs implies parabolic uniformly rectifiability. This converse
is a  consequence of more general results established in  \cite{BHHLN2}.  In combination, \cite{BHHLN1} and
 \cite{BHHLN2} prove that, just as in the elliptic
 setting of \cite{DS} and \cite{DS1},  we can characterize
 parabolic uniform rectifiability in terms of the existence
of a corona decomposition with
respect to an appropriate family of graphs ({\em regular} Lip(1,1/2) graphs).
In addition we obtain that all sufficiently ``nice" parabolic singular integral operators are $L^2$
bounded on a parabolic uniformly rectifiable set.

It is true that  in \cite{RN1, RN2,RN3},  the author took on the ambitious challenge to develop the theory of
parabolic uniformly rectifiable sets. Unfortunately though,  in \cite{RN1,RN2} the author either gives no proofs of his claims or supplies proofs which have gaps, a few of which, pertaining to \cite{RN1},
we pinpoint in \cite{BHHLN1}.
For now, let us point out three such errors or gaps in \cite{RN2}, as these are
directly relevant to the results in the present paper.    First, \cite[Lemma 6.2]{RN2} is essentially
our Theorem 2 stated above, and is stated in \cite{RN2} without proof, except for the claim that it is
essentially proved in \cite{HLN1}.  In fact, had that been the case, the authors of \cite{HLN1} would
have stated their results that way.  To be sure, our proof here follows that of \cite{HLN1} to some extent,
but an additional non-trivial idea, borrowed from \cite{DS}, is also used, in order to remove the extra flatness
assumption (mentioned above) imposed in \cite{HLN1}.  Second, \cite[Theorem 3.1]{RN2} is essentially
our Corollary 1 above, and relies on \cite[Lemma 5.3]{RN2}, which is a parabolic version of
a deep (elliptic) result of \cite{BL}.
However, the argument in \cite{RN2} relies on an application of Safonov's time-backwards
(i.e.,  non time-lagged) Harnack inequality
(see \cite{SY})
to solutions which do {\em not} vanish on $\Sigma$ (and thus
to which Safonov's result is inapplicable in any case), in a domain
which need not verify the Harnack Chain condition, a setting in which Safonov's result has not been proved.
Consequently, the proof of the parabolic version of the result of \cite{BL} (which may be found in
\cite{GH})
is rather more delicate than in the
elliptic case, as one is forced to account for the time lag in the parabolic Harnack inequality.
Finally, in \cite[Theorem 6.1]{RN2}, there is a claim (without proof)
that a 2-sided corkscrew condition yields big pieces of Lip(1,1/2) graphs (and even interior big pieces),
via the method of \cite{DJ}, without any mention of time-synchronization (even in a weak sense).
It is not clear to the present authors
how such a result might be proved, given the distinguished nature of the time direction in parabolic problems.
Perhaps it is true, but a proof should be given. 
We have not
checked in detail the validity of the argument in \cite{RN3}, as the result claimed
there is proved using a method entirely different to ours in
 \cite{BHHLN2}.

In \cite{HLN1}, \cite{HLN2} the assumption that $\Sigma$ is Reifenberg flat in the parabolic sense was motivated by  the  particular applications considered  but  this assumption may often seem too restrictive in other contexts. Therefore in \cite{NS} the fifth author, together with M. Str{\"o}mqvist, set out on the path to find and develop the parabolic analogue of the result of  G. David and D. Jerison \cite{DJ} mentioned above. In \cite{NS} it is proved that if $\Sigma\subset\mathbb R^{n+1}$ is a closed set which is Ahlfors-David regular in the parabolic sense, see Definition \ref{def1.ADR} below, and if  $\Sigma$
 satisfies what the authors called a synchronized two cube condition, then $\Sigma$ contains uniform big pieces of Lip(1,1/2) graphs by adapting the original arguments of \cite{DJ}.

To elaborate on the synchronized two cube condition, if $\Sigma  \subset \mathbb R^{n+1}$ is parabolic  Ahlfors-David regular  in the sense of Definition \ref{def1.ADR}, then $\Sigma$ is said to satisfy a {synchronized two cube condition} with constant $\gamma_1\in (0,1)$  if there exist, for all $(X,t)\in \Sigma$, $T_0<t<T_1$ and $0<r<\diam\Sigma$, two parabolic cubes $Q_\rho(X_1,t_1)$, $Q_\rho(X_2,t_2)$,  both contained in $Q_r(X,t)$, such that $Q_\rho(X_1,t_1)\cap(\R^{n}\times(T_0,T_1))$ and $Q_\rho(X_2,t_2)\cap(\R^{n}\times(T_0,T_1))$ belong to different connected components of $\mathbb R^{n+1}\setminus\Sigma$, and
\begin{eqnarray}\label{sync}
\gamma_1r\leq\rho<r,\ t_1=t=t_2.
\end{eqnarray}
Note that the condition as stated in \eqref{sync} is quite rigid as  the two cubes  $Q_\rho(X_1,t_1)$, $Q_\rho(X_2,t_2)$ have to satisfy
$t_1=t=t_2$, where $t$ is the time component of the original point $(X,t)$ fixed on the boundary. A more flexible condition would be to relax \eqref{sync} and to assume that $\Sigma$ instead satisfies a {weak synchronized two cube condition} with constant $\gamma_1$, in the sense that there exist, for all $(X,t)\in \Sigma$, $T_0<t<T_1$ and $0<r<\diam\Sigma$, two parabolic cubes $Q_\rho(X_1,t_1)$, $Q_\rho(X_2,t_2)$, as above and both contained in $Q_r(X,t)$, but with \eqref{sync} replaced by
\begin{eqnarray}\label{sync+}
\mbox{$\gamma_1r\leq\rho<r$, $t_1=t_2$}.
\end{eqnarray}
\eqref{sync+} is weaker compared to \eqref{sync} as the cubes $Q_\rho(X_1,t_1)$, $Q_\rho(X_2,t_2)$ still have to have the same time coordinate but this coordinate makes no explicit reference to time coordinate of the original point $(X,t)$ fixed on the boundary.

The discussion of the weak synchronized two cube condition leads us to the main
contributions of this paper.
First, assuming that $\Sigma\subset\mathbb R^{n+1}$ is a closed set which is parabolic
Ahlfors-David regular, and satisfies the general
(i.e., not necessarily synchronized) 2-cube (i.e., corkscrew)
condition, we prove that certain natural additional geometrical assumptions
imply a self-improvement of the corkscrew property, namely that if
 in addition $\Sigma$ is either
 time-forwards Ahlfors-David regular, time-backwards Ahlfors-David regular, or
parabolic uniform rectifiable, then in fact $\Sigma$ satisfies the
weak time-synchronized two cube condition discussed above.
Second, we show that
the results of \cite{NS} continue to hold with the strong time synchronized two cube condition replaced by the weak version;
more precisely, using the weak synchronized two cube condition, and revisiting the argument in \cite{NS},
we are able to establish uniform big pieces of Lip(1,1/2) graphs. Third, assuming
that $\Sigma  \subset \mathbb R^{n+1}$
is parabolic uniform rectifiable and satisfies the weak synchronized two cube condition we
are able to establish not only
uniform big pieces of Lip(1,1/2) graphs but also uniform big pieces of {\em regular}
Lip(1,1/2) graphs. This is what we need from the perspective of parabolic singular integrals and parabolic measure. Note that the latter conclusion was also established  (in partial form) in \cite{NS}, where the final part of the argument was left out and the authors referred to the corresponding arguments in \cite{HLN1}. Strictly speaking, the argument referred to in \cite{HLN1} applies only if the norm of the Carleson measure underlying the notion of parabolic uniform rectifiability is sufficiently small, depending on the dimension n and the constant defining the Ahlfors-David regularity, and thus, the proof in \cite{NS} applies in the presence of such a size restriction.  In this paper we remove this size restriction, and spell out the details of the argument using a parabolic version of a summation approach introduced in [DS]. Note that if, as in \cite{HLN1}, $\Sigma$ has the separation property and is
$\delta$-Reifenberg flat, then $\Sigma$ satisfies a synchronized two cube condition. This implication can
not be reversed. Hence, in particular and as already noted in \cite{NS}, our result generalizes
Theorem 1 in \cite{HLN1} beyond the hypothesis of $\Sigma$ being Reifenberg flat.

In addition, we give a parabolic counterpart of
 the main result in \cite{AHMNT} by proving
 that if $\Omega\subset\mathbb R^{n+1}$ is a domain defined as a connected component of $\mathbb R^{n+1}\setminus\Sigma$, if $\Om$ is a one-sided parabolic chord arc domain (see
 Definition \ref{Chord2}), and if  $\Sigma$ is parabolic uniformly rectifiable, then $\Om$ is in fact a
 parabolic chord arc domain (see Definition \ref{Chord3}). To prove this we use  \cite[Theorem 4.16]{BHHLN2} and
 \cite[Theorem 4.15(iii)]{BHHLN2}, and hence also \cite{BHHLN1}, to first conclude that if $\Sigma$ is parabolic uniformly rectifiable, then $\Sigma$ satisfies the parabolic bilateral weak geometric lemma, from which
 we then deduce the existence of exterior
 corkscrew points (and hence the chord-arc condition) more or less as in the elliptic case treated in
 \cite{AHMNT}, using the Harnack chain condition.

 The rest of the paper is organized as follows. In Section \ref{sec1} we introduce the
 geometric notions and terminology used in the paper.
 In Section \ref{sec2} we state the results proved in the paper: Theorems \ref{tftbtscorkscrews}-\ref{ParaAHMNT.thrm},
 and Theorems \ref{th1}-\ref{th2} with their respective corollaries.
 In particular, Theorems \ref{tftbtscorkscrews} and \ref{urtscorkscrews} give geometric criteria for the existence of
 weak time-synchronized corkscrew points,
 and Theorem \ref{ParaAHMNT.thrm} provides the geometric foundation for  Theorem \ref{th3+}.
 Theorem \ref{th1} (a precise version of ``Theorem 1" stated above), and Theorem \ref{th2}  and
 Corollary \ref{cor2} (together a precise version of ``Theorem 2" stated above) are the main results of the present work.
 In Section \ref{sec2}, we also briefly discuss, for the record, applications of our geometric results to the study of parabolic/caloric measure along the lines of \cite{NS} and \cite{GH}. In particular,
 we give Theorem \ref{NSweak}, which is the precise version of ``Corollary 1" stated above, and
we present
 Theorem \ref{th3+}, a
 precise version of ``Theorem 3" stated above.
Section \ref{sec3} is devoted to the proofs of Theorems \ref{tftbtscorkscrews}-\ref{ParaAHMNT.thrm} and Theorem \ref{th1} is proved in Section \ref{sec4}. The proof of Theorem \ref{th2} is given in Section \ref{sec5}.
In Section \ref{sec6}, we present two counter-examples to show that our weak time-synchronization
hypotheses are strict improvements over those
in \cite{NS}.

\section{Preliminaries and geometrical notions}\label{sec1}

 Points in  Euclidean space-time 
$ \mathbb R^{n+1} $ are denoted by $\Xb:= (X,t) = ( x_1,
 \dots,  x_n,t)$,  where $ X = ( x_1, \dots,
x_{n} ) \in \mathbb R^{n } $ and $t$ represents the time-coordinate. We will always assume that $n\geq 1$.  We let  $ \bar E, \ar E$,
  be the closure and boundary of the set $ E \subset
\mathbb R^{n+1}$. $  \lan \cdot ,  \cdot  \ran $  denotes  the standard inner
product on $ \mathbb R^{n} $ and we let  $  | X | = \lan X, X \ran^{1/2} $ be
the  Euclidean norm of $ X. $  We let $\|(X,t)\|:=|X|+|t|^{1/2}$ denote the parabolic length
of a space-time vector $\Xb =(X,t)$.
Given $(X,t), (Y,s)\in\mathbb R^{n+1}$,
we set
$$ 
d_p(X,t,Y,s): =\|(X,t)-(Y,s)\|= |X-Y|+|t-s|^{1/2}\,,$$
 and we define $ d_p ( X,t, E ) $  to equal the parabolic distance, defined with respect to $d_p(\cdot,\cdot)$, from
 $  (X,t) \in \mathbb R^{n+1} $ to $ E$.  We let
\[ Q_r ( X, t ) \, := \, \{ ( Y, s )\in \mathbb R^{ n + 1 } : | y_i  - x_i| < r, | t - s | < r^2  \},\]
 whenever $(X,t)\in
\mathbb R^{n+1}$, $r>0$, and we call $Q_r(X,t) $ a parabolic cube of ``length" $r$.
We may sometimes leave the center implicit, and write simply $Q_r$ to denote such a cube.
We also introduce the time-forward and time-backwards halves of $Q_r(X,t)$ as follows:
\begin{align*}
 Q_r^+ ( X, t )&:= Q_r( X, t )\cap \{ ( Y, s )\in \mathbb R^{ n + 1 } :s\geq t\},\notag\\
 Q_r^- ( X, t ) &:= Q_r( X, t )\cap \{ ( Y, s )\in \mathbb R^{ n + 1 } :s\leq t\}.
 \end{align*}
  We let $ \d x $ denote  Lebesgue $ n $-measure on    $ \mathbb R^{n}$ and
   given a number $\eta \geq 0$, we let $\H^\eta$ denote
 standard $\eta$-dimensional Hausdorff measure.
  We also define {\em parabolic} Hausdorff measure of {\em homogeneous}
  dimension $\eta$, denoted
  $\H_{\text p}^\eta$, in the same way that one defines standard Hausdorff measure, but using coverings
  by {\em parabolic} cubes, i.e., for $\delta>0$, and for $A\subset \R^{n+1}$, we set
  \[ \H_{\text{p},\delta}^\eta(A):= \inf \sum_k \diam_p(A_k)^\eta\,,
  \]
  where the infimum runs over all countable coverings of $A$, denoted $(A_k)_k$, with $\diam(A_k)\leq \delta$ for all $k$,
  and then define
  \[
  \H_{\text p}^\eta (A) := \lim_{\delta\to 0^+} \H_{\text{p},\delta}^\eta(A)\,.
  \]
As is the case for classical Hausdorff measure, $ \H_{\text{p}}^\eta$ is a Borel regular measure.
We refer the reader to \cite[Chapter 2]{EG2} for a discussion of the basic properties of standard
Hausdorff measure.  The arguments in \cite{EG2} adapt readily to treat $  \H_{\text{p}}^\eta$.
In particular, one obtains a measure equivalent to   $\H_{\text{p}}^\eta$ if one defines
$H_{\text{p},\delta}^\eta$ in terms of coverings by arbitrary sets of parabolic diameter at most $\delta$, rather than
cubes.  As in the classical setting, we define the parabolic homogeneous dimension of a set
$A\subset \R^{n+1}$ by
\[\H_{\text{p}, \text{dim}}(A):= \inf\left\{ 0\leq \eta<\infty \,| \,\H^\eta(A)=0\right\}\,.
\]
We observe that $\H_{\text{p}, \text{dim}}(\R^d)=d+1$; in particular $\H_{\text{p}, \text{dim}}(\ree)=n+2$.

Given a closed set $\Sigma \subset \mathbb R^{n+1}$
  of  homogeneous dimension $\H_{\text{p}, \text{dim}}(\Sigma)=n+1$, we
then define ``surface measure" on $\Sigma$ by
\begin{equation}\label{sigdef}
\sigma = \sigma_\Sigma:=  \H_{\text{p}}^{n+1}\lfloor_\Sigma\,.
\end{equation}
  We observe that this measure is apparently different to the one typically used in
  previous work on parabolic equations with time-varying boundaries;
  see, e.g., \cite{KW,LM,HL,HL1,HLN1,HLN2}. In those works, the following version
  of ``surface measure"  was used:  given a closed set $\Sigma \subset \mathbb R^{n+1}$,
for a Borel subset $E \subset \Sigma$,
 we set
 \begin{equation}\label{sigsdef}
 \si^{\bf s} ( E ):= \iint_E \, \d \si_t \,  \d t \,,
 \end{equation}
 where $\d\sigma_t $ denotes the restriction of $\H^{n-1}$
 to the time
slice $ E\cap (\mathbb R^{n} \times \{t\} )$.   It turns out that in the cases of greatest interest to us,
the ``slice" measure $\si^{\bf s}$, and the measure $\sigma$ defined in \eqref{sigdef}, are
equivalent (similar observations have been made previously in \cite{He} and \cite{MP}),
although they need not be equivalent in general.
\begin{remark}\label{r-measures}
Some further remarks are in order.
\begin{itemize}
\item[(i)] If $\sigma^{\bf s}$ (or for that matter
{\em any} measure $\mathfrak{m}$ defined on $\Sigma$) satisfies the parabolic
Ahlfors-David Regularity (p-ADR) condition (see Definition \ref{def1.ADR} below),
then so does $\si$, and in that case the two measures are of course equivalent.
This follows easily from the definition of
$\Hpn$ measure, and it is really just the same phenomenon that occurs in the classical (elliptic) case;
see \cite{DS}.

\smallskip

\item[(ii)] Consequently, if $\Sigma$ is a Lip(1,1/2) graph, then $\sigma \approx \sigma^{\bf s}$.
In particular, on a hyperplane $\pc\subset \ree$ parallel to the $t$-axis, which we may identify with
Euclidean space $\R^n$, we have that
$\H^n\lfloor_\pc \approx \Hpn\lfloor_\pc$, since the former is just $n$-dimensional Lebesgue measure
on $\pc$, which is parabolic ADR on $\Sigma = \pc$.

\smallskip

\item[(iii)]  If $\pc$ is a hyperplane parallel to the $t$-axis, and if $\pi$ is the
orthogonal projection operator onto $\pc$, then $\Hpn$ measure does not increase
under the action of $\pi$.  In particular, by virtue of item (ii), we have for any Borel set $A$
that $\H^n(\pi(A)) = \Hpn(\pi(A)) \leq \Hpn(A)$.

\smallskip

\item[(iv)]If  $\Sigma$ is parabolic uniformly rectifiable (p-UR; see Definition
\ref{def1.UR} below),
where we can initially define p-UR with respect either to $\sigma$, or to $\si^{\bf s}$,
then the two measures are equivalent.

\smallskip

\item[(v)] On the other hand,
the measures are not equivalent in general, even in the p-ADR setting.  In fact,
$\sigma^{\bf s} \lesssim \sigma$,
but the other direction does not need to hold.
\end{itemize}
Item (iii) follows exactly as in the classical case (see \cite[pp 75-76]{EG2}),
as one may readily verify using that the orthogonal projection operator is
Lipschitz with norm 1 with respect to the parabolic
metric, i.e., $\|\pi(X,t) - \pi(Y,s)\| \leq \|(X,t) -(Y,s)\|$.
Items (iv) and (v) are non-trivial.
We shall provide details of the proofs of the latter two facts in our forthcoming paper \cite{BHHLN1}.
See also \cite{He} and \cite{MP}.
\end{remark}

As above,
$\Sigma \subset \mathbb R^{n+1}$ will denote a closed set.
For  $ (X, t ) \in \Sigma $ and $r>0$, we shall
denote a ``surface cube" on $\Sigma$ by
$$\Delta(X,t,r):=\Sigma\cap Q_r(X,t)\,,$$
and its time-forward and time-backward halves by
\begin{align*}
 \Delta^+(X,t,r)&:= \Delta(X,t,r))\cap \{ ( Y, s )\in \mathbb R^{ n + 1 } :s\geq t\},\notag\\
\Delta^-(X,t,r) &:= \Delta(X,t,r))\cap \{ ( Y, s )\in \mathbb R^{ n + 1 } :s\leq t\}.
 \end{align*} The extremal time coordinates of $\Sigma$ will be denoted by $T_0=\inf\{t:\exists (X,t)\in\Sigma\}$ and $T_1=\sup\{t:\exists (X,t)\in\Sigma\}$. When we consider an open set
 $\Omega\subset \ree$, 
 we shall define $T_0$ and $T_1$ relative to $\Sigma= \partial\Omega$.

 Given a set $A\subset \ree$, we denote its topological interior by ${\tt int}(A)$.

\subsection{Parabolic Ahlfors-David regular sets}

\begin{definition}\label{def1.ADR}{\bf (Parabolic Ahlfors-David Regularity).} Let
$\Sigma \subset \mathbb R^{n+1}$ be a closed set.  We say that a measure $\mathfrak{m}$ defined
on $\Sigma$ is {\em parabolic Ahlfors-David regular}, parabolic ADR for short
(or simply p-ADR, or just ADR)
with constant $M\geq 1$,  if
\begin{equation} \label{eq1.ADRha-general}
M^{-1}\, r^{n+1} \leq \mathfrak{m}(\Delta(X,t,r)) \leq M\, r^{n+1},\end{equation}
whenever $0<r<\diam{\Sigma}$,  $(X,t)\in \Sigma$, $T_0<t<T_1$ and where $\diam{\Sigma}$ is the
(parabolic) diameter of $\Sigma$ (which may be infinite).
As noted above (see Remark \ref{r-measures} (i)), if
\eqref{eq1.ADRha-general} holds for any measure $\mathfrak{m}$ on $\Sigma$, then it holds for
$\sigma$ as in \eqref{sigdef}, i.e. for a possibly different but still universal choice of $M$,
\begin{equation} \label{eq1.ADRha}
M^{-1}\, r^{n+1} \leq \sigma(\Delta(X,t,r)) \leq M\, r^{n+1},\end{equation}
and in this case we simply say that
$\Sigma$ is  parabolic ADR (p-ADR, or just ADR).
\end{definition}
\begin{definition}\label{def.TBTFTSADR}{\bf (Time-Forward/Time-Backward/Time-Symmetric ADR).}
Let $\Sigma \subset \mathbb{R}^{n+1}$ be a closed set which is parabolic ADR as in Definition
\ref{def1.ADR} above.  We say that $\Sigma$ is {\em parabolic time-forward ADR},
or TFADR for short, if $T_1 = \infty$ and there exists a uniform constant  $M'\geq 1$, such that
 for each $(X,t) \in \Sigma$ with $T_0<t$ we have
 \begin{align*}
 \sigma(\Delta^+(X,t,r)) \geq  (M')^{-1} r^{n+1}.
 \end{align*}
Similarly, we say that $\Sigma$ is {\em parabolic time-backward ADR}, or parabolic TBADR for short, if $T_0 = -\infty$ and there exists a uniform constant $M'\geq 1$, such that
 for each $(X,t) \in \Sigma$ with $t<T_1$ we have
 \begin{align*}
 \sigma(\Delta^-(X,t,r)) \geq (M')^{-1}  r^{n+1}.
 \end{align*}
 If $\Sigma$ is both time-forwards ADR and time-backwards ADR, we say that $\Sigma$ is {\em time-symmetric ADR}
 (TSADR for short).
\end{definition}

\begin{definition}\label{def.dyadiccube}{\bf (Dyadic Cubes on an ADR Set).}
If $\Sigma$ is ADR, then  $(\Sigma,d_p,d\sigma)$ is a space of homogeneous type $\Sigma$ and as such
admits a parabolic dyadic decomposition (see \cite {Ch} for the construction, as well as \cite{HK}
for an alternative approach; the original construction, in the elliptic ADR setting, appears in  \cite{D1}, \cite{D2}).
That is, there exists a
constant $ \alpha>0$  such that for each $k \in \mathbb{Z}$
there is a collection of Borel sets, $\mathbb{D}_k$,  which we will call (dyadic) cubes, such that
$$
\mathbb{D}_k:=\{\mathcal{Q}_{j}^k\subset\Sigma: j\in \mathfrak{I}_k\},$$ where
$\mathfrak{I}_k$ denotes some  index set depending on $k$ (if $\Sigma$ is unbounded, then we may simply take
$\mathfrak{I}_k$ to be the set of positive integers, for each $k$), satisfying
\begin{eqnarray*}\label{cubes}
(i)&&\mbox{$\Sigma=\cup_{j}\mathcal{Q}_{j}^k\,\,$ for each
$k\in{\mathbb Z}$.}\notag\\
(ii)&&\mbox{If $m\geq k$ then either $\mathcal{Q}_{i}^{m}\subset \mathcal{Q}_{j}^{k}$ or
$\mathcal{Q}_{i}^{m}\cap \mathcal{Q}_{j}^{k}=\emptyset$.}\notag\\
(iii)&&\mbox{For each $(j,k)$ and each $m<k$, there is a unique
$i$ such that $\mathcal{Q}_{j}^k\subset \mathcal{Q}_{i}^m$.}\notag\\
(iv)&&\mbox{$\diam\big(\mathcal{Q}_{j}^k\big)\lesssim 2^{-k}$.}\notag\\
(v)&&\mbox{$\mathcal{Q}_{j}^k\supset \Sigma\cap Q_{\alpha2^{-k}}(Z^k_{j},t^k_{j})$ for some
$(Z^k_{j},t^k_j)\in\Sigma$ (the ``center" of $\qc_j^k$).}
\end{eqnarray*}
The dyadic cubes also enjoy a ``thin boundary property", but we shall not make use of that
fact in the present work.
\end{definition}

\begin{remark}\label{cuberemark} To avoid possible confusion, let us note that
we shall deal with four sorts of parabolic cubes in the sequel,
each with distinct notation:
the cubes $Q_r = Q_r(X,t)\subset \ree$, and the surface cubes $\Delta = \Delta(X,t,r):= Q_r(X,t) \cap\Sigma$,
defined above;  the {\em dyadic} ``cubes" on $\Sigma$, as in Definition \ref{def.dyadiccube}, which we denote by
the calligraphic $\mathcal{Q}$, and finally, $n$-dimensional parabolic cubes, defined on the hyperplane
$\pc:= \R^{n-1}\times \{0\} \times \R\cong \R^n$, which we define analogously to $Q_r$ in one less spatial dimension,
and which we denote by $I_r=I_r(x,t)$ for $(X,t) = (x,0,t)\in \pc$ (equivalently $I_r(x,t):= Q_r(x,0,t)\cap \pc$).

Mildly abusing notation, we write $\ell(Q_r):=r$, $\ell(I_r): =r$, $\ell(\Delta(X,t,r)):=r$,
and $\ell(\mathcal{Q}):= 2^{-k}$ when $\mathcal{Q} \in \mathbb{D}_k$.

We shall also use the letter $I$, and sometimes $J$,
to denote a {\em dyadic} parabolic cube in $\P\cong \R^n$; in particular,
such a cube has dimensions
$2^m\times...\times 2^m\times 2^{2m}$ for some integer $m$, and in this case we write $\ell(I) = 2^m$.
We apologize for the fact that this notation for side length
differs from that for the cubes $I_r$, by a factor of 2.
\end{remark}

\subsection{Parabolic uniform rectifiability}

\begin{definition}\label{def1.UR-} Assume that $\Sigma  \subset \mathbb R^{n+1}$ is parabolic ADR  in the sense of Definition \ref{def1.ADR}. Let
\begin{eqnarray*} \beta( Z, \tau, r  ) := \inf_P  \biggl ( r^{ - n -1 } \, \iint_{  \Delta ( Z, \tau,r) }  \, \biggl (\frac {d ( Y,s, P )}{r}\biggr )^2  \d \sigma (Y, s )\biggr )^{1/2}, \end{eqnarray*}
whenever $(Z,\tau)\in \Sigma $, $r>0$, and  where the infimum is taken with respect to all   $ n $ dimensional planes $ P $ containing a line
parallel to the $ t $ axis. Let
\begin{eqnarray}\label{me} \d \nu ( Z, \tau,
r  )  \, :=  \, \beta^2  ( Z, \tau, r) \,\d\sigma ( Z, \tau) \, r^{ - 1 }
\d r. \end{eqnarray}
We say that $ \nu$  is a Carleson measure on $    \Delta( Y,s,R )  \times ( 0, R ) $ if there exists $ \tilde M <
\infty $ such that
\begin{eqnarray}\label{1.9}
 \nu ( \Delta( X, t,\rho )  \times ( 0, \rho) ) \, \leq \,
\tilde M  \, \rho^{ n + 1 },
\end{eqnarray}
 whenever $ ( X, t  ) \in \Sigma  $ and $ Q_\rho ( X, t ) \subset Q_R ( Y, s )$.  The least such $ \tilde M  $ in \eqref{1.9} is
called the Carleson norm of $\nu$ on  $\Delta( Y,s,R ) \times ( 0, R ) $.
\end{definition}

\begin{definition}\label{def1.UR}{\bf (Parabolic Uniform Rectifiability).}  Assume that $\Sigma  \subset \mathbb R^{n+1}$ is parabolic ADR  in the sense of Definition \ref{def1.ADR} with constant $M$. Let $\nu$ be defined as in \eqref{me}. Then
 $\Sigma$ is {\em parabolic Uniformly Rectifiable},  parabolic UR (or simply p-UR) for short,  with UR constants $(M,\tilde M)$ if
\begin{eqnarray}\label{eq1.sf}
\| \nu \|:=\sup_{(X,t)\in\Sigma,\ \rho>0}  \rho^{ -n - 1 }\nu ( \Delta(X,t,\rho) \times ( 0, \rho) ) \, \leq \,
\tilde M.
\end{eqnarray}
\end{definition}

\subsection{Corkscrews and the  weak time-synchronized two cube condition} In the following definitions, Definitions \ref{corkscrews}-\ref{tsynk+}, we consistently assume that
$\Sigma  \subset \mathbb R^{n+1}$ is a closed set.

\begin{definition}\label{corkscrews}{\bf (Corkscrew, 2-Cube Condition).} Let $\gamma_0\in (0,1)$ be given. We say that $\Sigma$ satisfies a
{\em corkscrew condition} (more precisely, {\em 2-sided corkscrew condition}, or {\em 2-cube condition})
with constant $\gamma_0$, if there exists, for all $(X,t) \in \Sigma$, $T_0<t<T_1$ and $0<r<\diam\Sigma$, two parabolic cubes  $Q_\rho(X_1,t_1)$, $Q_\rho(X_2,t_2)$, both contained in $Q_r(X,t)$, such that $Q_\rho(X_1,t_1)\cap(\R^{n}\times(T_0,T_1))$ and $Q_\rho(X_2,t_2)\cap(\R^{n}\times(T_0,T_1))$ belong to different connected components of $\mathbb R^{n+1}\setminus\Sigma$, and with $$\gamma_0 r \leq \rho < r.$$
\end{definition}

\begin{definition}\label{tsynk}{\bf (Weak Time-Synchronized 2-Cube Condition).}
 Let $\gamma_1\in (0,1)$ be given. We say that $\Sigma$ satisfies a weak time-synchronized two cube condition with constant $\gamma_1$, if there exist, for all $(X,t)\in \Sigma$, $T_0<t<T_1$ and $0<r<\diam\Sigma$, two parabolic cubes $Q_\rho(X_1,t_1)$, $Q_\rho(X_2,t_2)$, both contained in $Q_r(X,t)$, such that $Q_\rho(X_1,t_1)\cap(\R^{n}\times(T_0,T_1))$ and $Q_\rho(X_2,t_2)\cap(\R^{n}\times(T_0,T_1))$ belong to different connected components of $\mathbb R^{n+1}\setminus\Sigma$, and with
$$\gamma_1r\leq\rho<r\,,\quad t_1=t_2\,.$$
\end{definition}

\noindent{\em Remark}.  The (strong) synchronized 2-cube condition considered in \cite{NS}
entailed the further requirement that the cubes $Q_\rho(X_1,t_1)$ and $Q_\rho(X_2,t_2)$ be synchronized also
with $Q_r(X,t)$, i.e., $t_1=t=t_2$.

\begin{definition}\label{corkscrews+-}{\bf (Interior Corkscrew Condition).} Let $\gamma_0\in (0,1)$ be given.  Let
$\Omega\subset\mathbb R^{n+1}$ be an open set with boundary $\pom=\Sigma$.
We say that $\Omega$ satisfies an {\em interior corkscrew condition} with constant $\gamma_0$, if there exists, for all $(X,t) \in \Sigma$, $T_0<t<T_1$ and $0<r<\diam\Sigma$, a parabolic cube  $Q_\rho(X_1,t_1)$, contained in $Q_r(X,t)$, such that $Q_\rho(X_1,t_1)\cap(\R^{n}\times(T_0,T_1))\subset\Omega$  and with $$\gamma_0 r \leq \rho < r.$$
\end{definition}

\begin{definition}\label{corkscrews+}{\bf (Corkscrew Condition w.r.t. an open set $\Om$).}
 Let $\gamma_0\in (0,1)$ be given.  Let
$\Omega\subset\mathbb R^{n+1}$ be an open set with boundary $\pom=\Sigma$.
  We say that $\Om$ (or sometimes, in keeping with previous terminology,
  $\partial\Omega$) satisfies a {\em corkscrew condition} (more precisely {\em
  2-sided corkscrew condition})
  with constant $\gamma_0$, if there exists, for all $(X,t) \in \Sigma$, $T_0<t<T_1$ and $0<r<\diam\Sigma$, two parabolic cubes  $Q_\rho(X_1,t_1)$, $Q_\rho(X_2,t_2)$, both contained in $Q_r(X,t)$, such that $Q_\rho(X_1,t_1)\cap(\R^{n}\times(T_0,T_1))\subset\Omega$ and  $Q_\rho(X_2,t_2)\cap(\R^n\times(T_0,T_1))\subset\mathbb R^{n+1}\setminus
  \overline{\Omega}$, and with $$\gamma_0 r \leq \rho < r.$$
\end{definition}

\begin{definition}\label{tsynk+}{\bf (Weak Time-Synchronized 2-Cube Condition w.r.t. an open set).}
Let $\gamma_1\in (0,1)$ be given.
Let $\Omega\subset\mathbb R^{n+1}$ be an open set with boundary $\pom=\Sigma$.
We say that $\Om$ (or sometimes, in keeping with previous terminology,
  $\partial\Omega$)  satisfies a {\em weak time-synchronized two cube condition}
  with constant $\gamma_1$, if there exist, for all $(X,t)\in \partial\Omega$, $T_0<t<T_1$ and $0<r<\diam\Sigma$, two parabolic cubes $Q_\rho(X_1,t_1)$, $Q_\rho(X_2,t_2)$, both contained in $Q_r(X,t)$, such that $Q_\rho(X_1,t_1)\cap(\R^n\times(T_0,T_1))\subset\Omega$, $Q_\rho(X_2,t_2)\cap(\R^n\times(T_0,T_1))\subset\mathbb R^{n+1}\setminus \overline{\Omega}$, and with
  $$\gamma_1r\leq\rho<r\,,\quad t_1=t_2\,.$$
\end{definition}

\noindent{\em Remark}.  We observe that in Definition \ref{corkscrews+} (resp.  \ref{tsynk+}), we are assuming that
$\Sigma =\pom$ satisfies Definition \ref{corkscrews} (resp.,  \ref{tsynk}), but with the additional requirement that one of the stipulated components of $\ree\setminus \pom$ lies in $\Om$, at every scale and at every boundary point.

\subsection{Harnack chains and parabolic chord arc domains} In the following definitions, Definitions \ref{Chord1}-\ref{Chord3}, we consistently assume that
$\Sigma  \subset \mathbb R^{n+1}$ is a closed set and that
$\Omega\subset\mathbb R^{n+1}$
is a connected open set (a domain) with boundary $\pom=\Sigma$.
In addition we will simply assume $\diam\Sigma=\infty$, $T_0=-\infty$ and $T_1=\infty$, to avoid tedious notation. If $T_0$ or $T_1$ is finite, the interested reader can formulate the localized versions of the definitions.
\begin{definition}\label{Chord1}{\bf (Harnack Chain condition).}
We say that $\Om$ is Harnack chain connected
(or that it satisfies the {\em Harnack Chain condition})
with constants $\kappa > 100$ and $C_* \ge 1$ if the following holds.
For every $(U_1,s_1), (U_2,s_2) \in \Om$, with $$(s_2 -s_1)^{1/2} \ge \kappa^{-1} d_p((U_1,s_1), (U_2,s_2))$$ there exists a chain of parabolic cubes $\{Q_i\}_{i = 1}^\ell$, $Q_i = Q_{r_i}(X_i,t_i)$, $i = 1,2,\dots, \ell$ with $Q_i \in \Om$, such that
\begin{align*}
(i)&\mbox{ $(U_1,s_1) \in Q_1$ and $(U_2,s_2) \in Q_\ell$},\\
(ii)&\mbox{ $Q_{i + 1} \cap Q_i \neq \emptyset$, $i = 1, 2,\dots, \ell -1$},\\
(iii)&\mbox{ $(C_*)^{-1}\diam(Q_i) \le d(Q_i,\partial \Om) \le C_* \diam(Q_i)$},\\
(iv)&\mbox{ $t_{i +1} - t_i \ge (C_*)^{-1} r_i^2$ and}\\
(v)&\mbox{ the length of the chain, $\ell$, satisfies $\ell \le C_* \log_2\left( 2 + \frac{d((U_1,s_1), (U_2,s_2))}{\min_{i =1,2}d((U_i,s_i), \partial \Omega)} \right)$.}
\end{align*}
\end{definition}

\begin{definition}\label{Chord2}{\bf (One-sided parabolic chord-arc domain).}
We say that $\Om$ is a one-sided parabolic chord arc domain with constants $(M,\gamma_0,\kappa,C^\ast)$ if
\begin{align*}
(a)&\mbox{ $\partial \Om$ is parabolic Ahlfors-David regular with constant $M$},\\
(b)&\mbox{ $\Om$ satisfies the interior corkscrew condition with constant $\gamma_0$},\\
(c)&\mbox{ $\Om$ satisfies the Harnack chain condition with constants $(\kappa,C_*)$}.
\end{align*}
\end{definition}

\begin{definition}\label{Chord3}{\bf (Parabolic chord-arc domain).}
We say that $\Om$ is a  parabolic chord arc domain with constants $(M,\gamma_0,\kappa,C^\ast)$ if
\begin{align*}
(a)&\mbox{ $\partial \Om$ is parabolic Ahlfors-David regular with constant $M$},\\
(b)&\mbox{ $\partial\Om$ satisfies the (two-sided) corkscrew condition with constant $\gamma_0$},\\
(c)&\mbox{ $\Om$ satisfies the Harnack chain condition with constants $(\kappa,C_*)$}.
\end{align*}
\end{definition}

Note that the only difference between Definition \ref{Chord2} and Definition \ref{Chord3} relates to the corkscrew conditions stated in $(b)$: in  Definition \ref{Chord2} only interior corkscrews in the sense of Definition \ref{corkscrews+-} are assumed while in  Definition \ref{Chord3} the (full) corkscrew condition in the sense of Definition \ref{corkscrews+} is assumed.

\subsection{Uniform big pieces} Assume that $\Sigma  \subset \mathbb R^{n+1}$ is parabolic ADR  in the sense of Definition \ref{def1.ADR}. Let in the following $\pi$ denote the orthogonal projection onto the plane $\{(x,x_n,t)\in\mathbb R^{n-1}\times\mathbb R\times\mathbb R:x_n=0\}$. At instances we identify $ \mathbb R^{n}$ with $\mathbb R^{n-1}\times \{0\}\times \mathbb R$, and  put
\[ I_r ( z, \tau  ) = \{ ( y, s ) \in \mathbb R^{n} :
| y_i   - z_i   | < r, \, i = 1,\dots, n - 1, \, \,   | s - \tau |
< r^2  \} \]
for $( z, \tau ) \in \mathbb R^{n}$, $r > 0$.

\begin{definition}\label{bigp1}{\bf (Uniform Big Pieces of Lip(1,1/2)  graphs).}   We say that $\Sigma$ contains
uniform big pieces of Lip(1,1/2)  graphs with constants $(\epsilon, b)$ if the following condition holds: Given $ ( X, t ) \in \Sigma$, $T_0<t<T_1$ and $0< R <\diam \Sigma, $ there exists, after a
possible rotation in the space variable, a Lip(1,1/2) function $ \psi $ with constant $b$, and $\epsilon>0$, such that
\begin{eqnarray}\label{aa1-}
 \H^{n}(\pi(\Sigma_\psi\cap\Delta( X, t,R))) \geq \epsilon R^{n + 1 },
 \end{eqnarray}
 where
  \begin{eqnarray}\label{aa2}
 \Sigma_\psi:=\{(x,x_n,t)\in\mathbb R^{n-1}\times\mathbb R\times\mathbb R:x_n=\psi(x,t)\}.
 \end{eqnarray}
 \end{definition}

 \begin{remark}
  Note that \eqref{aa1-} implies, as Hausdorff measure does not increase under projections, that
 \begin{eqnarray}\label{aa3}
 \si (  \Sigma_\psi\cap\Delta( X, t,R) ) \geq \epsilon R^{n + 1 }.
 \end{eqnarray}
 \end{remark}

 \begin{definition}\label{bigp2}{\bf (Uniform Big Pieces of $RPLip$  graphs).}   We say that $\Sigma$ contains
uniform big pieces of {\em regular parabolic} Lip(1,1/2)  ($RPLip$ for short)
graphs with constants $(\epsilon,b_1,b_2)$  if \eqref{aa1-} and \eqref{aa2} hold whenever $ ( X, t ) \in \Sigma$, $T_0<t<T_1$ and $0< R <\diam\Sigma$, but  with a regular parabolic Lip(1,1/2) function $ \psi $, satisfying \eqref{1.7} with constants $b_1$, $b_2$, and for $\epsilon>0$.
 \end{definition}

 \begin{definition}\label{bigp3}{\bf (Interior Big Pieces of Lip(1,1/2)   graphs).}
 Let $\Omega\subset\mathbb R^{n+1}$ be an open set with boundary $\pom=\Sigma$.
 We then say that $\partial\Omega$ satisfies a uniform interior big pieces of Lip(1,1/2)  graphs condition with
 constants
 $$\epsilon>0,\,\,b\geq 0,\,\,C\geq 1,\,\,c>0,\,\,A>0\,,$$
 if the following holds:  given $ ( \hat X, \hat t )=(\hat{x},\hat{x}_n,\hat{t}) \in \Omega$,
 we can find  a Lip(1,1/2) function $\psi$ with constant $b$, and
 a domain $\tilde\Omega$, 
 such that with $d:=d(\hat X,\hat t,\Sigma)$, we have
 \begin{itemize}
 \item[(i)] $Q_{\epsilon d}(\hat{X},\hat{t})\subset
 \tilde\Omega\subset \Omega\cap Q_{Cd}(\hat X,\hat t)$.
 \smallskip
 \item[(ii)] After a
possible rotation in the space variables we have
\[
\tilde\Omega=\{(y,y_n,s): (y,s)\in I_{cd}(x,t),\ \psi(y,s)<y_n< \hat{x}_n+Ad\},
\]
where
$( X, t)=( x, x_{n}, t)$ is some point in
$\Sigma\cap Q_{Cd}(\hat X,\hat t)$  with
\[
\Delta_{d/2} (X,t) \subseteq \Sigma\cap Q_{Cd}(\hat X, \hat t).
\]
 \item[(iii)] $ \H^{n}\left(\pi(\Sigma\cap\partial\tilde\Omega)\cap Q_{cd}( x, t)\right) \geq \epsilon d^{n + 1 }.$
 \end{itemize}
 \end{definition}

  \begin{definition}\label{bigp4}{\bf (Interior Big Pieces of $RPLip$  graphs).}
  Let $\Omega\subset\mathbb R^{n+1}$ be an open set with boundary $\pom=\Sigma$.
   We then say that $\partial\Omega$ satisfies a uniform interior big pieces of regular parabolic Lip(1,1/2)  graphs condition with constants $(\epsilon,b_1,b_2,C,c,A)$ if the following hold. Given $ ( \hat X, \hat t ) \in \Omega$, we can find a domain $\tilde\Omega$ such that
  $(i)-(iii)$ of Definition \ref{bigp3} hold for some regular parabolic Lip(1,1/2) function $ \psi $ with constants $(b_1,b_2)$ and for some constant $b=b(b_1,b_2)$.
 \end{definition}

\section{Statement of main results}\label{sec2}
We first prove the following two theorems concerning additional weak geometrical assumptions beyond  parabolic ADR  which  imply that $\Sigma$ satisfies the weak time-synchronized two cube condition. We consider these theorems  elementary but important.
\begin{theorem} \label{tftbtscorkscrews}
Let $\Sigma$ be a closed subset of $\R^ {n+1}$ which is parabolic ADR with constant $M$ and assume that $\Sigma$ satisfies a corkscrew condition in the sense of Definition \ref{corkscrews} with constant $\gamma_0$.
Assume, in addition, that $\Sigma$ is either time-forwards ADR with constant $M'$ or time-backwards ADR with constant $M'$.  Then $\Sigma$ satisfies the weak time-synchronized two cube condition in the sense of Definition \ref{tsynk} with $\gamma_1=\gamma_1(n,M,\gamma_0,M')$.  Furthermore,  given $(X,t)\in \Sigma$, $T_0<t<T_1$ and $0<r<\diam\Sigma$, and if $\Sigma$ is time-forwards ADR or time-backwards ADR, then the two synchronized cubes in the weak time-synchronized two cube condition can be constructed to be contained in $Q_r^{+}(X,t)$ and $Q_r^{-}(X,t)$, respectively.
\end{theorem}

\begin{theorem} \label{urtscorkscrews}
Let $\Sigma$ be a closed subset of $\R^ {n+1}$ which is parabolic ADR with constant $M$ and assume that $\Sigma$ satisfies a corkscrew condition in the sense of Definition \ref{corkscrews} with constant $\gamma_0$.
Assume, in addition, that  $\Sigma$ is parabolic UR in the sense of
Definition \ref{def1.UR-}  with UR constants $(M,\tilde M)$.  Then $\Sigma$ satisfies the weak time-synchronized two cube condition in the sense of Definition \ref{tsynk} with $\gamma_1=\gamma_1(n,M,\gamma_0,\tilde M)$.
\end{theorem}

We are able to prove the following parabolic counterpart of the result in \cite{AHMNT}.

\begin{theorem}\label{ParaAHMNT.thrm}
Let $\Om$ be a one-sided parabolic chord arc domain with constants $(M,\gamma_0,\kappa,C^\ast)$. If, in addition,
$\Sigma$ is parabolic uniformly rectifiable  with constants $(M,\tilde M)$, then $\Om$ is a parabolic chord arc domain with constants $(M,\hat\gamma_0,\kappa,C^\ast)$, where $\hat\gamma_0=\hat\gamma_0(n,M,\tilde M,\gamma_0,\kappa,C^\ast)$.
\end{theorem}

Concerning uniform big pieces of Lip(1,1/2)  graphs we prove the following.

\begin{theorem}\label{th1} Let $\Sigma$ be a closed subset of $\R^ {n+1}$ which is parabolic ADR with constant $M$. Assume that $\Sigma$ satisfies the weak time-synchronized two cube condition in the sense of Definition \ref{tsynk} with $\gamma_1\in (0,1)$.  Then $\Sigma$ contains
uniform big pieces of Lip(1,1/2) graphs with constants $(\epsilon,b)$ depending only $n,M$ and $\gamma_1$.
\end{theorem}

\begin{corollary}\label{cor1} Let $\Sigma$ be a closed subset of $\R^ {n+1}$ which is
parabolic ADR with constant $M$,
and let $\Omega\subset\mathbb R^{n+1}$ be an open set
with boundary $\pom=\Sigma$.
 Assume that $\partial\Omega$ satisfies a corkscrew condition in the sense of Definition \ref{corkscrews+} with
 constant $\gamma_0$ and that  $\partial\Omega$ is time-symmetric ADR in the sense of Definition \ref{def.TBTFTSADR}
with constant $M'$. Then $\partial\Omega$ satisfies a uniform interior
big pieces of Lip(1,1/2)  graphs condition with constants $(\epsilon,b,C,c,A)$ depending only on $n,M,\gamma_0$ and $M'$.
\end{corollary}

Concerning uniform big pieces of regular parabolic Lip(1,1/2)  graphs we prove the following.

\begin{theorem}\label{th2} Let $\Sigma$ be a closed subset of $\R^ {n+1}$ which is parabolic UR with constants $(M,\tilde M)$,
and which satisfies the corkscrew condition in the sense of Definition \ref{corkscrews}
with constant $\gamma_0$. 
Then $\Sigma$ contains
uniform big pieces of RPLip  graphs with constants  $(\epsilon,b_1,b_2)$ depending only on
$n,M,\tilde M$ and $\gamma_0$.
\end{theorem}

\begin{corollary}\label{cor2} Let $\Sigma$ be a closed subset of $\R^ {n+1}$ which is
parabolic UR with constants $(M,\tilde M)$.
 let $\Omega\subset\mathbb R^{n+1}$ be an open set
with boundary $\pom=\Sigma$.
Assume that $\Omega$ satisfies a corkscrew condition in the sense of Definition \ref{corkscrews+}
with constant $\gamma_0$ and that  $\partial\Omega$ is time-symmetric ADR in the sense of
Definition \ref{def.TBTFTSADR}
with constant $M'$.  Then $\partial\Omega$ satisfies a uniform interior  big pieces of  
RPLip graphs condition  with constants $(\epsilon,b_1,b_2,C,c,A)$
depending only on for some $n,M,\tilde M,\gamma_1$ and $M'$.
\end{corollary}

Naturally Theorems \ref{tftbtscorkscrews}-\ref{ParaAHMNT.thrm} and Theorems \ref{th1}-\ref{th2}, along
with their corollaries, have  applications to the study of parabolic/caloric measure.
Given an open set $\Omega \subset \ree$, and a point $(X,t) \in \Omega$, we let
$\omega(X,t,\cdot)$ denote caloric measure for $\Omega$ with pole at $(X,t)$.
Then in particular, combining Corollary \ref{cor2}
with the results of \cite{GH}, we have the following.  For simplicity, we state the result in the case that
$\diam\Sigma=\infty$, $T_0=-\infty$ and
$T_1=\infty$.  In the case that
$T_0$ or $T_1$ is finite, one may modify the formulation appropriately; see \cite{GH}.
\begin{theorem}\label{NSweak} Let $\Omega \subset \ree$ and $\Sigma = \pom$
be as in Corollary \ref{cor2}.
Then caloric measure is absolutely continuous with respect to $\sigma$,
and satisfies a local weak reverse H\"older condition.  More precisely,
there are constants $C\geq 1$, $\lambda >0$ depending on the constants in Corollary \ref{cor2},
such that, given $(X_0,t_0) \in \Sigma$ and
$r>0$, we have for every $(X,t) \in \Omega \setminus Q_{4r}(X_0,t_0)$ 
that $\omega(X,t,\cdot) \ll\sigma$ on $\Delta_r(X_0,t_0)=\Sigma \cap Q_r(X_0,t_0)$, with
$d\omega(X,t,\cdot)/d\sigma=:h$ satisfying
\begin{align}\label{weakRH}
\left(\rho^{-n-1}\iint_{\Delta_{\rho}(Y,s)}h^{1+\lambda}d\sigma\right)^{1/(1+\lambda)} &\leq
C\rho^{-n-1}
\iint_{\Delta_{2\rho}(Y,s)}h\, d\sigma 
\notag\\
&=C\rho^{-n-1}\omega(X,t,\cdot)\left(\Delta_{2\rho}(Y,s)\right),
\end{align}
whenever $(Y,s)\in\Sigma$ and $Q_{2\rho}(Y,s)\subset Q_{r}(X_0,t_0)$, where
$\Delta_{\rho}(Y,s)=Q_{\rho}(Y,s)\cap \Sigma$, and
$\Delta_{2\rho}(Y,s)=Q_{2\rho}(Y,s)\cap \Sigma$.  Equivalently, we obtain solvability of the
Dirichlet problem\footnote{See \cite{GH} for the precise formulation of the $L^p$ Dirichlet problem
(and initial-Dirichlet problem in the case that $T_0$ is finite).}
with $L^p$ (lateral) boundary data, for some $p<\infty$.
\end{theorem}

We remark that the results in \cite{GH} are stated and proved with underlying measure
$\sigma$ given by our measure $ \si^{\bf s}$ defined as in \eqref{sigsdef}, however, all the arguments in
\cite{GH} carry over with this measure replaced by our $\sigma$ measure defined as in \eqref{sigdef}.

Next, we state another application, in the context of parabolic chord arc domains.
To set the stage, let $\Sigma$ be a closed subset of $\R^ {n+1}$ which is  parabolic ADR with
constant $M$, let $\Omega\subset\mathbb R^{n+1}$ be a connected component of
$\mathbb R^{n+1}\setminus\Sigma$ and assume that $\diam\Sigma=\infty$, $T_0=-\infty$ and $T_1=\infty$.
Using the Wiener criterion in \cite{EG}\footnote{In the initial discussion of the Dirichlet problem in Section 4 in \cite{NS}
the correct assumption is of course that  $\Sigma $ should be parabolic time-backward parabolic ADR, not only parabolic ADR. Indeed, if  $\Sigma $ is parabolic time-backward parabolic ADR and $\Omega\subset\mathbb R^{n+1}$ is a connected component of $\mathbb R^{n+1}\setminus\Sigma$, then the uniform capacity estimate stated in \cite{NS} can be verified and using the Wiener criterion in \cite{EG} one can conclude that if $\diam\Sigma=\infty$, $T_0=-\infty$ and $T_1=\infty$, then any point $(X,t)\in\partial\Omega$ is regular for the bounded continuous Dirichlet problem for the heat equation in $\Omega$.} we can conclude that any point $(X,t)\in\partial\Omega$ is regular for the bounded continuous Dirichlet problem for the heat equation, as well as the adjoint heat equation, in $\Omega$.  Using this and exhausting $\Omega$ by bounded sets, and applying Perron-Wiener-Brelot type arguments, one can conclude that the bounded continuous Dirichlet problems for the heat
equation, as well as the adjoint heat equation, in $ \Om $ always have unique solutions.

Recall that $ \om( \hat X,\hat t, \cdot)$ is the caloric  measure,
 at $(\hat X,\hat t)\in\Omega$, associated to the heat equation in $ \Omega$.
For $(X,t)$,  $r>0$, and $A\geq 100$, we
define
\begin{eqnarray}\label{2.6}
\Gamma^+_A(X,t,r)&=&\{(Y,s): | Y - X|^2 \leq  A ( s - t ),\,\,  s
- t  \geq  5 r^2\}.
\end{eqnarray}

\begin{definition}\label{def-Ainftyball}  Let $(X,t)\in\partial\Omega$, $r>0$,  and consider
$(\hat X,\hat t)\in \Omega\cap\Gamma_A^+(X,t,4r)$.  We say that
$\omega(\cdot)=\omega(\hat X,\hat t,\cdot)$ satisfies a
{\em Reverse H\"older} condition (equivalently the $A_\infty$ condition)
on $\partial\Omega\cap Q_r(X,t)$, with constants $L$ and $\lambda>0$ if the following is true:
$\omega$ is a doubling measure, i.e.,
\[\omega(Q_{2\rho}(\tilde X,\tilde t)) \lesssim \omega(Q_{\rho}(\tilde X,\tilde t))\,,
\]
and  $\d\omega/\d\sigma=h$ exists on $\Delta(X,t,r) $ with
\begin{eqnarray}\label{cc1}\iint_{\Delta(\tilde X,\tilde t,\rho)}h^{1+\lambda}\, \d\sigma\leq L\sigma(Q_\rho(\tilde X,\tilde t))^{-\lambda}(\omega(
Q_{\rho}(\tilde X,\tilde t)))^{1+\lambda}
\end{eqnarray}
whenever $(\tilde X,\tilde t)\in \partial\Omega$, $Q_{2\rho}(\tilde X,\tilde t)\subset Q_r(X,t)$.
\end{definition}

The following theorem is an immediate consequence of the combination
of Theorem \ref{ParaAHMNT.thrm} (which gives the (2-sided)
corkscrew condition), Theorem \ref{urtscorkscrews} (which
gives the weak time-synchronized two cube condition), Corollary \ref{cor2} (
which gives the uniform interior  big pieces of $RPLip$ graphs condition),
the doubling property of parabolic measure (which can be
proved as in \cite{HLN2}), and a familiar argument based on the maximum principle.
We note that the following is a parabolic analogue of the main result of \cite{HM},
although our approach is based on the much more efficient method of \cite{AHMNT}, using big pieces technology.

\begin{theorem}\label{th3+}
Suppose that $\Om$ is a one-sided parabolic
chord arc domain with constants $(M,\gamma_0,\kappa,C^\ast)$, and with boundary $\pom=:\Sigma$.
Assume also
that $\diam\Sigma=\infty$, $T_0=-\infty$ and $T_1=\infty$,and that $\Sigma$ is parabolic uniformly rectifiable
with constants $(M,\tilde M)$.
Let $(X,t)\in\partial\Omega$, $r>0$, $A\geq 100$,  and consider
$(\hat X,\hat t)\in \Omega\cap\Gamma_A^+(X,t,4r)$. Then $ \om ( \hat X ,
\hat t, \cdot ) $ is a doubling measure in the sense that there exists a
constant $c=c(n,M,\tilde M, \gamma_0,\kappa,C^\ast, A) $  such that
\begin{eqnarray}\label{1.6}
 \om ( \hat X, \hat t, \Delta(\tilde X,\tilde t,2\rho)) \leq c\om ( \hat X, \hat t, \Delta(\tilde X,\tilde t,\rho)),
\end{eqnarray}
for all $(\tilde X,\tilde t)\in\partial\Omega$, $Q_{\rho}(\tilde X,\tilde t)\subset Q_{2r}(X,t)$. Furthermore,
$\omega(\hat X,\hat t,\cdot)$ satisfies the Reverse H\"older condition (i.e., the
$A_\infty$ condition) on $\Delta(X,t,r)$ in the sense of Definition \ref{def-Ainftyball} with constants
$L$ and $\lambda>0$ depending only on $(n,M,\tilde M, \gamma_0,\kappa,C^\ast, A)$.
\end{theorem}

In the following sections we give the proofs of Theorems \ref{tftbtscorkscrews}-\ref{ParaAHMNT.thrm}, and
Theorems \ref{th1}-\ref{th2} with their corollaries, in the case that $\diam\Sigma=\infty$, $T_0=-\infty$ and
$T_1=\infty$. If $T_0$ or $T_1$ is finite, the proofs are  completely analogous and in this case the difference is that all sets occurring have to be intersected with $\R^n\times(T_0,T_1)$ and the notation will be more cumbersome.

We conclude this section by
 recalling the following elementary lemma from \cite{GH} which we shall use in the sequel.
\begin{lemma}[\cite{GH}] \label{GHlemma}
 Let $\Sigma$ be a closed subset of $\R^ {n+1}$ which is parabolic ADR with constant $M$.   Assume that $\Sigma$ is  time-forward ADR or time-backwards ADR  with constant $M'$. Then there exist constants $a_1 \in(0,1/2)$, $a_2\in(0,1)$, both depending only on $n,M$ and $M'$ such that the following is true. Let
 $(X,t)\in \Sigma$. If $\Sigma$ is  time-forward ADR, then
 $$\sigma(\Delta^+(X,t,r)\cap \{(Y,s): s < t + (a_1r)^2\}) \geq a_2r^{n+1} $$
 and if $\Sigma$ is  time-backwards ADR, then
 $$\sigma(\Delta^-(X,t,r)\cap \{(Y,s): s < t - (a_1r)^2\}) \geq a_2r^{n+1}.$$
\end{lemma}
\begin{proof} See \cite{GH}.
\end{proof}

\section{Proof of the geometric theorems}\label{sec3}  In this section we prove Theorem \ref{tftbtscorkscrews}, Theorem \ref{urtscorkscrews}, and Theorem \ref{ParaAHMNT.thrm}; when $\diam\Sigma=\infty$, $T_0=-\infty$ and $T_1=\infty$.

\subsection{Proof of Theorem \ref{tftbtscorkscrews}} Let $\Sigma$ be a closed subset of $\R^ {n+1}$ which is parabolic ADR with constant $M$ and assume that $\Sigma$ satisfies a corkscrew condition in the sense of Definition \ref{corkscrews} with constant $\gamma_0$.
Assume, in addition, that $\Sigma$ is  time-backwards ADR with constant $M'$.
That case in turn implies the time-forward case, by the change of variable $t \to -t$.
Our goal is to prove that $\Sigma$ satisfies the weak time-synchronized two cube condition in the
sense of Definition \ref{tsynk} with $\gamma_1=\gamma_1(n,M,\gamma_0,M')$.

Let $(X,t)\in \Sigma$, $r>0$. We first apply the corkscrew condition at $(X,t)$ and on the scale $r/C_1$ where $C_1$ a large constant to be chosen, to produce two cubes $$Q_1:=Q_{\gamma_0 r/C_1}(Y_1,s_1),\ Q_2 := Q_{\gamma_0 r/C_1} (Y_2,s_2),$$
both contained in $Q_{r/C_1}(X,t)$,
but belonging to different connected components of $\mathbb R^{n+1}\setminus\Sigma$.
If $s_1 =s_2$, then we are done and hence  we can without loss of generality assume that $s_1 < s_2$.
Let $\delta := s_2-s_1$.

Assume that $\delta \leq(\gamma_0r/2C_1)^2$. In this case it follows readily that we can find two cubes $Q_1'$ and $Q_2'$, both of size
 $\gamma_0r/(2C_1)$, $Q_1'\subset Q_1$,  $Q_2'\subset Q_2$,  such that the centers of  $Q_1'$ and $Q_2'$ have the same time coordinate and we are done.

 Assume that $\delta > (\gamma_0r/2C_1)^2$. Using that $Q_1$ and $Q_2$ are contained in
different connected components of $\mathbb{R}^{n+1}\setminus\Sigma$ we see that the line connecting $(Y_1,s_1)$ and $(Y_2,s_2)$
intersects $\Sigma$ at some point $(Z_1,\tau_1)\in \Sigma$.  Set $\delta ' = \tau_1-s_1$.
Our strategy is now to use Lemma \ref{GHlemma} to produce a chain of cubes, starting with a cube centred at $(Z_1,\tau_1)$, such that the terminal cube in the chain has time coordinate very close to $s_1$.  To start the construction of the chain we let
$$\Delta_1:=\Delta(Z_1,\tau_1,\gamma_0r/C_2)$$
where $C_2>C_1$ is yet an other large constant to be chosen.  Applying Lemma \ref{GHlemma} to $\Delta_1$ we can pick
$$(Z_2,\tau_2)\in\Delta_1 \cap \{(X,t) \in \Sigma: t<\tau_1 -a^2\gamma_0^2r^2/(C_2^2)\}$$
Where $a$ is the constant denoted by $a_1$ in the statement of Lemma \ref{GHlemma}.   Also, let
 $$\Delta_2:=\Delta(Z_2,\tau_2,\gamma_0r/C_2).$$
 We can now repeat this argument with $(Z_1,\tau_1)$ replaced by $(Z_2,\tau_2)$ to iteratively produce a sequence of points $(Z_i,\tau_i) \in \Sigma$ and we let $N$ be the first integer such that $|\tau_N - s_1|< (\gamma_0r/C_2)^2$.
 At $(Z_N,\tau_N)$ we apply the corkscrew condition at scale
$\gamma_0r/(2C_2)$ to produce a corkscrew cube $Q_0$ centered at $(Y_0,s_0)$, contained in a component of $\mathbb{R}^{n+1}$ which is different  the component containing $Q_1$,
of parabolic size $\gamma_0^2r/(2C_2)$, and such that $Q_0\subset Q_{\gamma_0r/(2C_2)}(Z_N,\tau_N)$. Then, as in the case  $\delta \leq(\gamma_0r/2C_1)^2$ it follows readily  that $Q_1$ contains a cube of size $\gamma_0^2r/(2C_2)$ with the same time coordinate as $Q_0$.

To complete the proof it only remains to show how to choose $C_1$ and $C_2$ appropriately to ensure that $Q_0$ lies inside $Q_r(X,t)$.  However, it is easy to see that
\begin{align*}
N \frac{a^2\gamma_0^2 r^2}{C_2^2} \leq s_1-s_2 \implies N \leq \frac{C_2^2(s_1-s_2)}{a^2\gamma_0^2r^2} \leq \frac{4C_2^2}{C_1^2\gamma_0^2a^2}.
\end{align*}
This implies that
\begin{align*}
\|Z_1-Z_N\| &\leq \frac{4C_2^2}{C_1^2a^2} \frac{\gamma_0r}{C_2} = \frac{4C_2\gamma_0 r}{C_1^2a^2}\mbox{ and }\|Z_1 - X\| \leq \frac{r}{C_1}.
\end{align*}
Hence, if  we choose $C_1>100$, and $C_2 = \max\{C_1+1,C_1^2a^2/(40\gamma_0)\}$ (here we can make $\gamma_0$ smaller than $a$ if necessary), then $ \|Z_N-X\| \leq r/50$ and consequently
$$\|(X-Y_0,t-s_0)\| \leq r/25. $$
This proves that $Q_0 \subset Q_r(X,t)$.

To see that the corkscrew cube constructed can be constructed as to be contained in $Q_r^-(X,t)$ we first apply  Lemma \ref{GHlemma} and then repeat the same argument above, but
with $(X,t,r)$ replaced by $(X',t',r')$ where  $(X',t') \in \Delta^-(X,t,r/100)$ and where $r'=r'(a_1,r)$ is chosen so that $\Delta(X',t',r') \subseteq \Delta^-(X,t,r/100)$. This completes the proof of  Theorem \ref{tftbtscorkscrews}.
\subsection{Proof of Theorem \ref{urtscorkscrews}} We
 introduce for $ ( Z, \tau) \in \Sigma, r > 0$,
 \begin{eqnarray}\label{2.1}\beta_\infty ( Z, \tau, r ) := \inf_P
\, \sup_{  ( Y,s ) \in \De  ( Z, \tau, r  )}  \frac{
 d ( {Y,s}, P )  }{r},
 \end{eqnarray}
 where  the infimum is  taken over all $ n $-planes $ P $
containing a line parallel to the $ t $ axis. Given $ ( Z, \tau),
r  $ as above, in display (2.2) in \cite{HLN1} it is proved that
 \begin{eqnarray}\label{2.2}
  \beta_\infty(Z,\tau,r)^{n+3} \leq 16^{n+3}\beta^2(Z,\tau,2r).
   \end{eqnarray}
   We also consider the dyadic versions
\begin{equation}\label{3}
\beta_\infty(\qc):=  \inf_P \diam(\qc)^{-1} \sup_{\{(Y,s)\in k\qc\}} \dist(Y,s,P)\,,
\end{equation}
and
\begin{equation}\label{4}
\beta(\qc) = \beta_2(\qc):=  \inf_P \left(\diam(\qc)^{-d-2} \int_{2k\qc} \dist^2(Y,s,P) \,d\sigma(Y)\right)^{1/2}\,,
\end{equation}
where $\qc$ is a dyadic cube as in
Definition \ref{def.dyadiccube},  $k$ is a sufficiently large number to be chosen,
and for $k\geq 1$
we define the ``dilate" $k\qc:= \{(Y,s)\in\Sigma: \,\dist\left(Y,s, \qc\right)\leq k\diam(\qc)\}$.
We then have the dyadic version of \eqref{2.2}, by the same argument:
 \begin{eqnarray}\label{2.2a}
  \beta_\infty(\qc)^{n+3} \leq C \, \beta^2(\qc)\,,
   \end{eqnarray}
   where $C=C(n,ADR)$.

   By definition, since $\Sigma$ is p-UR, we have that
   $\beta^2(X,t,r) d\sigma(X,t) dr/r$ is a Carleson measure on $\Sigma \times (0,\infty)$, which
   readily implies in turn
(in fact is equivalent to) the fact that $\beta(\qc)$ satisfies the dyadic Carleson measure condition
\begin{equation}\label{dyadicbeta}
\sup_\qc \frac1{\sigma(\qc)} \sum_{\qc'\subset \qc} \beta^2(\qc')\,
\sigma(\qc')=:\|\beta\|_{\mathcal{C}}<\infty\,.
\end{equation}
Using \eqref{2.2a}, one may readily verify (basically via Tchebychev's inequality)
that \eqref{dyadicbeta} implies
a Carleson packing condition for  ``non-flat" cubes, as follows:
given $\varepsilon>0$, there is a constant $C_\varepsilon<\infty$ such that
\begin{equation}\label{badpack}
\sup_\qc \frac1{\sigma(\qc)} \sum_{\qc'\subset \qc} \alpha_\varepsilon(\qc') \leq C_\varepsilon\,,
\end{equation}
where
\[
\alpha_\varepsilon(\qc'):=\left\{
\begin{array}{lc}
\sigma(\qc')\,, & \text{if }\, \beta_\infty(\qc')\geq \varepsilon,
\\[6pt]
0\,,&\text{if }\, \beta_\infty(\qc')< \varepsilon\,.
\end{array}
\right.
\]
Consider a cube $Q_r(X,t)$ centered on $\Sigma$.
By the standard properties of the dyadic system, there is a dyadic cube
\[
\qc_0 \subset \Delta(X,t,r/10) =Q_{r/10}(X,t)\cap\Sigma\,,
\]
with $\ell(\qc_0)\approx r$.
Fix $\ep$ suitably small to be chosen, and note that
as a consequence of the packing condition \eqref{badpack}, there is a dyadic
subcube $\qc_1 \subset \qc_0$ with
\[ c_\ep r \,\leq \,  r_1:= \diam(\qc_1) \,<\, r/100\,,\]
such that $\beta_\infty(\qc_1)<\ep$.  Fixing $\Xb^1=(X^1,t^1)=(x^1,x_n^1,t^1)\in \qc_1$,
we see that by the
definition of dyadic $\beta_\infty$, see \eqref{3}, there is a hyperplane $P_1$
parallel to the $t$-axis such that
\begin{equation}\label{sandwich}
\dist(Y,s,P_1)<\eps r_1\,,\quad \forall \,(Y,s)\in  \Delta_1:= Q_{10r_1}(\Xb^1) \cap\Sigma\,.
\end{equation}
provided that $k$ is chosen large enough, depending on the constants in the construction of
the dyadic system in Definition \ref{def.dyadiccube}.
By translation we may
suppose that $\Xb^1=(0,0)$, and by a spatial rotation we may suppose that
$P_1 = \R^{n-1}\times \{0\}\times \R$.
Set $Q_1:= Q_{10r_1}(\Xb^1) = Q_{10r_1}(0,0)$, define
\[Q_1^{up}:= Q_1 \cap \{y_n \geq \ep r_1\}\,,\quad Q_1^{down}:= Q_1 \cap \{y_n \leq - \ep r_1\}
\]
and observe that $Q_1^{up}\cap\Sigma =\emptyset = Q_1^{down}\cap\Sigma$, by \eqref{sandwich}.
By the (2-sided) corkscrew condition (Definition \ref{corkscrews}), we see that
$Q_1^{up}$ and $Q_1^{down}$ lie in distinct connected components of
$\ree\setminus \Sigma$, call them $\Omega^+$ and $\Om^-$ respectively,
provided that we fix $\ep$ small enough depending
on the constant $\gamma_0$ in Definition \ref{corkscrews}.  In particular, we choose
$\ep<1$, and then define
\[Q^{\pm}:= Q_{r_1}(0,\pm 3r_1,0) \subset \Om^\pm \cap Q_1\,.
\]
Since $Q_1\subset Q_r(X,t)$, and  $r_1\approx r$,
the conclusion of Theorem \ref{urtscorkscrews} follows.

\subsection{Proof of Theorem \ref{ParaAHMNT.thrm}} The proof of Theorem \ref{ParaAHMNT.thrm} has similarities with the proof of Theorem \ref{urtscorkscrews}.  Let $k\geq 2$.
We introduce the {\em bilateral}
dyadic $\beta_\infty$ numbers
\begin{equation*}
b\beta_{\infty}(\qc) :=
\diam(\qc)^{-1}\inf_{P}\left\{  \sup_{\Yb \in k\qc} \dist(\Yb,P)
\, + \sup_{ \Zb \in P \cap B(\Xb_\qc,\,k\diam(\qc))}  \dist(\Zb,\Sigma) \right\}\,,
\end{equation*}
where $\Xb_\qc$ is the ``center" of the dyadic cube $\qc\subset \Sigma$,
as in  Definition \ref{def.dyadiccube} $(v)$.
We say that
$\Sigma$ satisfies the bilateral weak geometric lemma with parameter $\epsilon$,
if there exists $M_\epsilon > 0$ such that for every dyadic cube $\mathcal{R} \in \mathbb{D}(\Sigma)$,
\[\sum_{\substack{\qc \subseteq \mathcal{R}\\ b\beta_{\infty}(\qc) > \epsilon}} \sigma(\qc) \,\le \,
\, M_\epsilon \,\sigma(\mathcal{R}).  \]
Since $\Sigma$ is parabolic UR we can apply  \cite[Theorem 4.16]{BHHLN2} and
\cite[Theorem 4.15(iii)]{BHHLN2} to conclude that $\Sigma$ satisfies the parabolic
bilateral weak geometric lemma, for every fixed $\ep>0$,
where $k\geq 2$ is at our disposal, and will
eventually be chosen large enough.   We now follow one of the two arguments in \cite{AHMNT}.

Let $Q_r(X,t)$ be centered on $\Sigma$, and let $\ep>0$ be a sufficiently small number
to be chosen.  Following the proof of Theorem \ref{urtscorkscrews} in the preceding subsection,
we may again construct a dyadic cube $\qc_1$,
of diameter $r_1 \approx r$, for which now
$b\beta_{\infty}(\qc_1)<\ep$,  along with slightly modified versions of the
cubes $Q^\pm$
as above, still disjoint from $\Sigma$, and contained in the same cube $Q_1$ as before,
but now off-set in time, so that
\[Q^{\pm}:= Q_{r_1}(0,\pm 3r_1, \pm r_1^2) \,.
\]
In addition, 
by the interior corkscrew condition, choosing $\ep$ small enough
we may assume without loss of generality that $Q^+\subset \Omega$.  If $Q^-$ lies in a different
connected component of $\ree \setminus \Sigma$ than does $Q^+$, we are done.
Otherwise if both $Q^\pm\subset \Omega$, then by the Harnack Chain condition
we may connect the points $\Yb^\pm := (0,\pm 3r_1, \pm r_1^2)$ by a chain of cubes $\{Q^m\}_m$
of uniformly bounded cardinality, with
\[ Q^m\subset \Omega \cap Q_{Cr_1}(0,0)\,,\,\, \text{ and }
\,  \ell(Q^m)\approx \dist(Q^m,\Sigma) \geq c r_1\,,
\]
for every $m$, and
with $c,C$ each depending on the constants in the
Harnack Chain condition.  For $k\gg C$, and $\ep\ll c$, we contradict
the fact that $b\beta_{\infty}(\qc_1)<\ep$.
The proof of Theorem \ref{ParaAHMNT.thrm} is complete.

\section{The proof of Theorem \ref{th1} and Corollary \ref{cor1}} \label{sec4}

 We here prove Theorem \ref{th1} and Corollary \ref{cor1}. We will give the proofs only in the case  when $\diam\Sigma=\infty$, $T_0=-\infty$ and $T_1=\infty$. Throughout the section we assume that  $\Sigma$ is a closed subset of $\R^ {n+1}$, which is parabolic ADR with constant $M$, and we assume that $\Sigma$ satisfies the weak time-synchronized two cube condition in the sense of Definition \ref{tsynk} with $\gamma_1\in (0,1)$.

 It is true that the proof of Theorem \ref{th1} has substantial overlap with the corresponding result in \cite{NS} and the difference is that in our proof we have to be even more careful  as we only assume that  $\Sigma$  satisfies the {weak  time-synchronized two cube condition} while in \cite{NS} it is assumed that  $\Sigma$  satisfies the {synchronized two cube condition}. For the convenience of the reader we in the following give what we believe is a  sufficiently detailed presentation of the proof of  Theorem \ref{th1} and we try to highlight the key differences in the proof compared to \cite{NS}.

 We have divided our presentation into three subsections, Subsections 5.1-5.3.  In Subsection 5.1 we reduce the proof of  Theorem \ref{th1} to Proposition \ref{mainprop}. In Subsection 5.2 we prove  Corollary \ref{cor1} and in Subsection 5.3 we prove Proposition \ref{mainprop}.

 \subsection{Reducing  Theorem \ref{th1} to Proposition \ref{mainprop}}
 The argument in this subsection follows closely its counterpart in \cite{DJ}, but of course
 adapted to the parabolic setting.
 We start by redefining
\begin{eqnarray}\label{redef}\mbox{$M$ to equal  $\max\{M,\sqrt{n}\gamma_1,4n$\}}.
\end{eqnarray}
Based on \eqref{redef} we can without loss of generality assume that $\Sigma$ is parabolic ADR with constant $M$ and that there exist,
 for all $(X,t)\in \Sigma$ and $R>0$, two parabolic cubes $Q_\rho(X_1,t_1)$, $Q_\rho(X_2,t_2)$, both contained in $Q_R(X,t)$ but belonging to different connected components of $\mathbb R^{n+1}\setminus\Sigma$, and with
\begin{eqnarray}\label{syncmod}
\rho=M^{-1}R\,,\quad  t_1=t'=t_2.
\end{eqnarray}
Consider the points $(X_1,t' )$, $(X_2,t')$, and consider the line in $\mathbb{R}^n\times \{t'\}$ 
connecting $(X_1,t')$
and $(X_2,t')$. As
$(X_1,t' )$, $(X_2,t')$  belong to different connected components of $\mathbb R^{n+1}\setminus\Sigma$,
this line meets $\Sigma$ at a point which we denote by $(X',t')$. Let $\delta_i:=||X_i-X'||$, $i=1,2$, and note that $M^{-1}R\leq \delta_i\leq R$. We will construct the big piece of Lip(1,1/2) graph to be contained
in the set of points on $\Sigma$ which are reached by lines emanating from points in $Q_\rho(X_1,t_1)$ and which are parallel to the line connecting
$X_1$ and $X_2$. It is clear that we can translate and re-scale our setting about the point $(X',t')$ and in particular we can in the following and without loss of generality assume that  $$\mbox{$R= 2M$ and $(X',t')=(0,0)$.}$$
Through this $(X_1,t' )$, $(X_2,t')$ are mapped to $(Y_1,0)$, $(Y_2,0)$, the corkscrew cubes $Q_\rho(X_1,t_1)$, $Q_\rho(X_2,t_2)$ are mapped to $Q_2(Y_1,0)$, $Q_2(Y_2,0)$, and
$\Sigma$ is mapped to a new closed set having the same quantitative properties as $\Sigma$: for simplicity we will, with an abuse of notation, also use the notation $\Sigma$ for this set.

Consider the time-independent hyperplane $\pc$ which passes through $(0,0)$ and is orthogonal
to $(Y_1,0)$. Then by construction we can  after a possible rotation in the spatial coordinates,
represent points in $\mathbb R^{n+1}$ as
$\mathbf{X}=(x,x_n,t)\in\mathbb R^{n-1}\times\mathbb R\times\mathbb R $, and in this
coordinate system $\mathbf{Y}_1:=(Y_1,0)$ and $\mathbf{Y}_2:=(Y_2,0)$ are  represented by
$$\mathbf{Y}_1=(0,\overline{M}_1,0),\ \mathbf{Y}_2=(0,\overline{M}_2,0),\mbox{ respectively},$$
where  $2 \leq \overline{M}_i \leq 2M$. We may then identify the hyperplane $\pc$ with
$\mathbb R^{n-1}\times \{0\}\times\mathbb R$. Let
$\pi$ denote the orthogonal projection onto this plane and let $\pi^\perp$ denote the
orthogonal projection onto the normal to the plane. Let $\mathcal{U}$ be the component of
$\mathbb{R}^{n+1} \setminus \Sigma$ containing $\mathbf{Y}_1$.

Given
$(z,\tau)\in  \mathbb{R}^n$ we let
\[
I_r(z,\tau) = \{(y,s)\in \mathbb{R}^n: |y_i-z_i| <  r, i=1,..,n-1, \indent |s-\tau| < r^2\}.
\]
Define $I^0 := \overline{I_1(0,0)}$, 
and set $\overline{M}:=\overline{M}_1$,
\begin{eqnarray}\label{icube}
I^{\overline{M}} : = \{\mathbf{X}: (x,t)\in I^0,\ x_n = \overline{M} \}.
\end{eqnarray}
By construction, $I^{\overline{M}}$ is a closed $n$-dimensional parabolic cube
contained in the same component as $Q_2(\mathbf{Y}_1)$ (namely
$\mathcal{U}$), and
$d(I^{\overline{M}},\Sigma) \geq 1$.  We also note that
$$
D := \pi\left(Q_{1/2}(\mathbf{Y}_2)\right) =\frac{1}{2}I^0,
$$
and $\sigma(D)=\mathcal{H}^n(D) = 2^{-n-1}$. In particular, choosing
\begin{eqnarray}\label{const1}
\gamma  = 2^{-n-2}\,,
\end{eqnarray}
we have
\begin{eqnarray}\label{ineq1}
\sigma(D)\ge 2\gamma.
\end{eqnarray}
Note that any line in the $x_n$ direction connecting $D \times \{x_n = -M\}$ with $I^{\overline{M}}$
has to intersect $Q_{1/2}(\mathbf{Y}_2)$ and $Q_2(\mathbf{Y}_1)$, thus it also has to intersect $\Sigma$.

Given $h>0$ we introduce
$$
\Gamma = \Gamma_h := \{\mathbf{X}\in\mathbb{R}^{n+1}:x_n \geq h\|(x,t)\|\},
$$
i.e. $\Gamma$ is a parabolic cone with aperture $h$, and we let
\begin{align}\label{Sdef}
S :=& \{\mathbf{X} \in \Sigma: -M \leq x_n \leq \overline{M},\mbox{ and if
$\mathbf{Y}\in \mathbf{X}+\Gamma$, $y_n =
	\overline{M}$, then ${\bf Y}\in I^{\overline{M}}$}\}.
\end{align}
Note that $S\subset\Sigma$, and that
$\pi(S)\subset I^0$.  Also,
if in  this construction we choose $h\geq 6M$, it follows that if $\mathbf{X}=(x,x_n,t) \in \Sigma$, $-M \leq x_n \leq \overline{M}$, and if $(x,t) \in D$, then
${\bf Y}\in I^{\overline{M}}$ whenever $\mathbf{Y} \in \mathbf{X} + \Gamma$
is such that $y_n = \overline{M}$.   Indeed, for such a point $\mathbf{Y}$, we have
\begin{align*}
3M \geq \overline{M} + M \geq y_n-x_n &\geq h\|(y,s) - (x,t)\|\geq h(\|(y,s)\|-\|(x,t)\|).
\end{align*}
Hence
\begin{align*}
 3M \geq h(\|(y,s) \|-{1}/2)
\implies 1 \geq \|(y,s)\|,
\end{align*}
as $h\geq 6M$ and the last conclusion in the display implies that $ (y,s) \in I^0$.
In particular,  $D\subset \pi(S)$ and thus by \eqref{ineq1},
 \begin{eqnarray}\label{eq5.6}
 \H^n(\pi(S)\cap I^0)=
\H^n(\pi(S))\ge 2\gamma.
\end{eqnarray}
 To prove  Theorem \ref{th1} it  suffices to prove the following proposition.
 \begin{proposition}\label{mainprop} Let $\gamma$ be as in \eqref{const1}, \eqref{ineq1}. Then there exists $h>0$, depending only on $n$ and $M$, such that if we let $\Gamma=\Gamma_h$, and if we define
\begin{eqnarray*}
W:=\{(x,t)\in I^0: \exists\ \mathbf{X}=(x,x_n,t)\in S,\ (\mathbf{X}+\Gamma)\cap S=\{\mathbf{X}\}\}
\end{eqnarray*}
(so that in particular, $W\subset \pi(S)$),
then  $\H^n(\pi(S)\setminus W)\le \gamma$.
\end{proposition}

We defer the proof of Proposition \ref{mainprop} until Subsection \ref{ss5.3} below.

\begin{remark}\label{r5.1} Let us record a remark summarizing the preceding observations.
Set
\[W':=\left\{{\bf X}=(x,x_n,t)\in S:\, (\mathbf{X}+\Gamma)\cap S=\{\mathbf{X}\}, \text{ and }
 \pi({\bf X})= (x,t) \in I^0\right\}
\]
(thus, $\pi(W') = W$), and define
\[\Om':= {\tt int}\left(\bigcup_{{\bf X}\in W'} (\mathbf{X}+\Gamma)\right) .
\]
Then
\[
\Om'\cap \{y_n< \overline{M}+ 2\}\subset \mathcal{U}
\]
(recall that $\mathcal{U}$ is the component of
$\mathbb{R}^{n+1} \setminus \Sigma$ containing $\mathbf{Y}_1$), and
$\pom'$ is given by a Lip(1,1/2) graph $\{(y,\psi(y,s),s)\}$, where $\psi$ has Lip(1,1/2) norm equal to $h$.
Note that $W' \subset \Sigma \cap \pom'$, and thus
\[\pi(\Sigma\cap \pom')\cap I^0\supset \pi(W')= W\,.
\]
Also, by Proposition \ref{mainprop}, we have $\H^n(\pi(S)\setminus W)\le \gamma$, and therefore by
\eqref{eq5.6},
\[
\H^n(\pi(\Sigma\cap \pom')\cap I^0)\geq\H^n(W)\geq \gamma
\]
Furthermore, if for some $N\geq 2$, we have that $Q_N({\bf Y}_1)\subset  \mathcal{U}$, then
\begin{equation}\label{eq57}
\Om'\cap \{y_n< \overline{M}+ N\}\subset \mathcal{U}\,.
\end{equation}
\end{remark}

Thus, taking Proposition \ref{mainprop} for granted,
we conclude that there is a Lip (1,1/2)  graph $G$ with constant $ h$ such that
$\H^n(\pi(\Sigma\cap G)\cap I^0)\ge \gamma$.
Thus, conditioned on Proposition \ref{mainprop} the proof of  Theorem \ref{th1} is complete.

Proposition \ref{mainprop} is essentially
Lemma 2.1 in \cite{NS}, and we again emphasize that the difference
now is that in the present proof of this key result
we assume only
that  $\Sigma$  satisfies the {\em weak}  time-synchronized two cube condition, while
in \cite{NS} it is assumed that  $\Sigma$  satisfies the (strong)
synchronized two cube condition. This weaker assumption will force us to revisit
certain subtleties of the proofs in \cite{NS} and \cite{DJ}.

 \subsection{Proof of Corollary \ref{cor1}} Let $\Omega\subset\mathbb R^{n+1}$ be an open
 set with $\pom = \Sigma$, and 
 assume that $\Omega$ satisfies a corkscrew condition in the sense of Definition \ref{corkscrews+}
 with constant $\gamma_0$ and that  $\Sigma$ is time-symmetric ADR in the sense of
 Definition \ref{def.TBTFTSADR}
with constant $M'$.  Consider  $(\hat{X},\hat{t}) \in \Omega$ and let  $(X',t') \in \partial \Omega$ be a
point such that
$$d_p(\hat{X},\hat{t},X',t') = d_p(\hat{X},\hat{t},\partial \Omega)=:d.$$
Our hypotheses are invariant under the change of variable $t\mapsto - t$,
so without loss of generality we may assume that $t' \geq \hat{t}$.

Let $N$ be a sufficiently large number to be chosen.  If $t'-\hat t \leq (N^{-2}h^{-1}d)^2$, then we
translate in time so that
$(X',t') = (X',0)$.  Otherwise, if $t'-\hat t > (N^{-2}h^{-1}d)^2$, then using TBADR, and iterating
Lemma \ref{GHlemma},
we may find $(X'',t'')\in \Sigma$ such that
\[d_p(\hat X,\hat t,X'',t'') \approx d
\]
(depending implicitly on the constants in Lemma \ref{GHlemma}), with
$|t''-\hat t| \leq (N^{-2}h^{-1}d)^2$.  In this case, we translate in time so that
$(X'',t'') = (X'',0)$.  We set $\widetilde X:= X'$ in the first case, and
$\widetilde X:= X''$ in the second.
In either case,
 upon application of the corkscrew condition at $(\widetilde X,0)$, we can produce a
 corkscrew cube
$$Q_0= Q_0(\Xb^0)\,,\quad \text{for some } \Xb^0=(X^0,t^0)\in \mathbb{R}^{n+1} \setminus \Sigma\,,$$
 of (parabolic) diameter $N^{-2}h^{-1}\gamma_0 d/100$, whose
 distance to $(\widetilde X,0)$
is no more than $N^{-2}h^{-1}d$, and
  which is contained in a connected component of $\mathbb{R}^{n+1} \setminus \Sigma$ that
  does {\em not} contain $(\hat{X},\hat{t})$.
  We define the point
  \[ \Xb^1= (X^1,t^1) := (\hat X,t^0)
  \]
and we construct the subcube
  \[ Q_1(X^1,t^1)=:Q_1\subset Q_{d}(\hat{X},\hat{t})\,,
  \]
  of (parabolic) diameter $N^{-2}h^{-1}\gamma_0 d/100$ (i.e., equal to that of $Q^0$).
Note that by construction, for $N$ large we have $(X^1,t^1)\in\Omega$, and in fact
\[ d_p(X^1,t^1,\hat{X},\hat{t})= |t^0-\hat t|^{1/2}\lesssim N^{-2}h^{-1} d\ll d\,.
\]
Since
$\Xb^1$ and $\Xb^0$ lie in different connected components of $\mathbb{R}^{n+1} \setminus \Sigma$,
the line connecting them meets $\Sigma$, say at the point $\Xb^2=(X^2,t^0)$ (here we are using that $t^1=t^0$),
and by a translation in the space
variables, we may suppose that $X^2=0$.
Let $\pc$ denote the hyperplane through $(X^2,t^0) = (0,t^0)$
orthogonal to the line joining $\Xb^1$ to $\Xb^0$,
and note that
since $t^1=t^0$, the plane $\pc$ is parallel to the $t$-axis.
 Letting $\pi$ denote projection onto $\pc$, we have by construction that
 $\pi(\Xb^1) =\pi(\Xb^0)$.
We perform a rotation in the spatial variables, so that
$\P = \mathbb R^{n-1}\times \{0\}\times\mathbb R$, and so that
in this new coordinate system, for $N$ large, 
\[
(\hat{X},\hat t\,)= (0,\hat{x}_n,\hat t) = (0, \kappa d,ad^2)\,,\quad \text{ with } \,\, \frac12\leq \kappa\leq \kappa_0\,,
\text{ and }
|a|\leq (N^2 h)^{-2}\,,
\]
where $\kappa_0$ is uniformly controlled from above, and
\[\Xb^2\cong \pi(\Xb^2)=\pi(\Xb^1) =\pi(\Xb^0) = (0,t_0)= (0,\tilde ad^2)\,,
\quad \text{ with }
|\tilde a|\lesssim (N^2 h)^{-2}\,,
\]
After making a possible slight adjustment in diameter, by a purely dimensional factor $c(n)$, we may assume that
$Q_0$ and $Q_1$ have been rotated so that their faces are parallel to the coordinate hyperplanes in the new
coordinate system.

 Clearly, there is a constant $C\geq 1$ such that $Q_d(\hat{X},\hat{t})$, $Q_0$ and $Q_1$ are all contained in 
 $Q_{Cd}(0,t_0)$.  Furthermore, we can view $Q_0$ and $Q_1$
 as weak time-synchronized corkscrew cubes relative to $Q_{Cd}(0,t_0)$, so
 using the boundary point $(0,t_0)$ in place of the origin,
 we can run the argument above (as in the proof of Theorem \ref{th1}),
 with corkscrew cubes $Q_0$ and $Q_1$ at point $(0,t_0)$ and scale $2d$,
  to obtain the interior domain (see Remark \ref{r5.1}):
 \begin{equation}\label{eq.tom}
 \tilde\Omega=\{(Y,s):(y,s)=\pi(Y,s)\in I_*\,, \,\,\psi(y,s)<y_n<\kappa_1d\}
 \subseteq \Omega,
 \end{equation}
 where $\psi$ is a Lip(1,1/2) function with norm $h$, and where
 \begin{equation}\label{eq.istar}
 I_*:=  \pi(Q_*)\,,\quad  Q_*:=Q_{(Nh)^{-1} d}\, (0,0,t^0)\,,\quad \kappa_1:= \kappa + c(n)/200
 \end{equation}
 (so that $\kappa_1d = \hat x_n + c(n) d/200$), and
  \begin{equation}\label{eq.tbp}
  \mathcal{H}^n(\pi(\Sigma \cap \partial \tilde{\Omega}) \cap I_*)\geq \epsilon d^{n+1}\,,
   \end{equation}
 for some $\epsilon = \epsilon(n,\gamma_0,M')$.
    We should note that, when running the argument as in the proof of Theorem \ref{th1},
 we perform a parabolic rescaling, and then we ``undo" the
 parabolic rescaling to obtain the set $\tilde{\Omega}$ above.
 We observe that by construction,
 \begin{equation}\label{eq.tcs}
 Q_{N^{-2}h^{-1}d}(\hat X, \hat t)\subset \tilde{\Omega}\,,
 \end{equation}
 provided that we choose $N$ large enough.

 This concludes the proof of Corollary \ref{cor1}.

 \begin{remark}\label{r5.2}
 For future reference, let us record some additional observations.  To begin,
letting $G$ denote the graph
 of $\psi$, we may find
a point ${\bf X}^*=(x^*,x_n^*,t^*)\in \Sigma \cap G$ such that $\pi({\bf X}^*) \in  I_*$:
just choose ${\bf X}^*$ in the un-rescaled version of the set $W'$ in Remark \ref{r5.1}.
Note that such an ${\bf X}^*$ lies below the bottom face of $Q_{c(n)d}(\hat X,\hat t)$
(by construction of $G$, since
${\bf X}^* \in \Sigma$), hence
we see that $x_n^* \leq (\kappa - c(n))d$.
Since $\diam(I_*) \lesssim (Nh)^{-1} d$, and since $\psi$ has Lip(1,1/2) norm $h$, we find that
\begin{equation}\label{eq.psiest}
\sup_{(y,s) \in 100I_* }\psi(y,s) \leq (\kappa - c(n) +CN^{-1}) d \leq  (\kappa - c(n)/2) d\,,
\end{equation}
for $N$ large enough, and therefore with $c_1=c(n)/2$, we have
\begin{equation}\label{eq5.8}
 \hat{x}_n - \sup_{(y,s) \in 100I_* }\psi(y,s) \geq c_1 d \approx N \diam(I_*)\,.
\end{equation}
We note also that $Q_*$ is centered on
 $\Sigma$, at $\Xb^2=(0,0,t_0)$).
\end{remark}

\subsection{Proof of Proposition \ref{mainprop}}\label{ss5.3}
We roughly follow the argument in \cite{DJ},  as adapted to the parabolic setting in \cite{NS}, with some
modest technical refinements to deal with the fact that the time-synchronization in our 2-cube condition holds
only weakly.
As above, we identify $\R^n$ with the hyperplane $\pc= \{x_n=0\}$.
For any $$\mathbf{X}\in \{\mathbf{X}=(x,x_n,t):\ (x,t)\in I^0,\ x_n\in[-M,\bar M]\},$$
we let $L(\mathbf{X})$  denote the open line segment in the $x_n$ direction which connects $\mathbf{X}$ to $(x,\bar M,t)$. If $\mathbf{X}\in\Sigma$, then the length of $L(\mathbf{X})$ is at least $d(I^{\bar M},\Sigma)\ge 1$.
Define $\tilde \G$ to be the closure of the set of all such points $\mathbf{X}\in\Sigma$ which satisfy $L(\mathbf{X})\cap\Sigma=\emptyset$, and, recalling that the set $S$ is defined in
\eqref{Sdef}, we let
$$\G:=\tilde \G\cap S\subset \Sigma.$$
Given $A\subset \R^n$, set
\[\nu(A):= \sigma\left(\pi^{-1} (A)\cap Q_{2M}(0,0)\right)\,,
\]
where again $\pi$ denotes the orthogonal projection onto
$\pc = \mathbb R^{n-1}\times \{0\}\times\mathbb R$,
and note that $\nu$ defines a Borel measure with total mass
\[
\|\nu\|\leq \sigma\left(Q_{2M}(0,0)\right) \leq C
\]
(since $\sigma$ is ADR).
For $(x,t)\in I^0$, define
\[
\mc (x,t)=\sup\left\{\frac{1}{\H^n(I)}\,\sigma(\pi^{-1}(I)\cap Q_{2M}(0,0)): \,I\text{ contains }(x,t)\right\}\,,
\]
where the supremum runs over all parabolic cubes $I\subset \pc$ with $(x,t)\in I$,
so that $\mc(x,t) =\mc \nu(x,t)$, the parabolic
Hardy-Littlewood maximal function of $\nu$.
We let $\nt$ be a suitably large constant to be chosen momentarily.
Then, by the standard weak-type bounds, we have
\begin{eqnarray*}\mbox{$\B:=\{(x,t)\in \R^{n}:\ \mc(x,t)\ge \nt\}$ satisfies $\H^n(\B)\leq C/\nt$},
\end{eqnarray*}
for some constant $C=C(n,M)\geq 1$.
Then for $\nt=\nt(n,M,\gamma)$ large enough, we have
\begin{eqnarray}\label{ineq2}
\H^n(\B)\le \gamma/2.
\end{eqnarray}
We fix $\nt$ with respect to \eqref{ineq2}. In particular, since $\gamma$ is a purely dimensional constant
previously fixed (see  \eqref{const1}, \eqref{ineq1})
 it follows that $\nt$ is
from now on a fixed constant depending only on $n$, $M$.

Having fixed $\gamma$ and $\nt$, there will appear, in the construction to be outlined, four important constants: $\Lambda_0$, $\Lambda_1$, $\Lambda_2$, and $\Lambda_3$, with $1\leq\Lambda_i<\infty$ for $i\in\{0, 1,2, 3\}$. In general all constants
appearing will depend at most on $n$, $M$, $\Lambda_0$, $\Lambda_1$, $\Lambda_2$, and $\Lambda_3$. We will choose the degrees of freedom $\Lambda_0$, $\Lambda_1$, $\Lambda_2$, and $\Lambda_3$ to depend only on $n$, $M$, $\gamma$ and $\nt$, and hence to depend only on $n$, $M$. Furthermore, $\Lambda_i$ for $i\in\{0, 1,2, 3\}$, will be chosen to be of the form $2^{N_i}$ for some integer $N_i\geq 1$.

We let, for $\Lambda_1$ fixed and as above, we choose a dyadic number $A=2^{k_0}$ large enough that
\begin{eqnarray}\label{const3}
A > 2M\Lambda_1
\end{eqnarray}
 With $A$ fixed, we define, for $j\in\{0,1,...\}$,
\begin{align*}
\Sigma_j=\{&(x,t)\in I^0:\text{there exist }\mathbf{X}\in\G\text{ and } \mathbf{Y}\in S\text{ such that }\\
&\mathbf{X}=(x,x_n,t), \mathbf{Y}\in \mathbf{X}+\Gamma\text{ and }A^{-j}\le y_n-x_n < A^{-j+1}\}.
\end{align*}
If $\mathbf{X}=(x,x_n,t)\in S$, there exists a maximal $\hat x_n$ such that $(x,\hat x_n,t)\in S$. This follows since $x_n\le \bar M$ if $(x,x_n,t)\in S$, $I^{\bar M}\subset \U$ and $S$ is closed. Thus $(x,\hat x_n,t)\in \G$, which
shows that $\pi(\G)=\pi(S)$. When $(x,t)\in \pi(S)\setminus W$ we have
$\left[(\mathbf{X}+\Gamma)\setminus\{{\bf X}\}\right]\cap S\neq \emptyset$
whenever $\mathbf{X}=(x,x_n,t)\in S$. In particular, this is true for $\hat{\mathbf{X}}=(x,\hat x_n,t)\in \G$,
the maximal point constructed above,
so there exists  $\mathbf{Y}\in S\setminus \{{\hat{\bf X}}\}$ such that $\mathbf{Y}\in \hat{\mathbf{X}}+\Gamma$.
By our restriction on $A$ we have $$\mbox{$\pi(S)\setminus W\subset \cup_j \Sigma_j$.}$$
Furthermore, as by construction $\H^n(\B)\le \gamma/2$, the proof of Proposition \ref{mainprop} is reduced to proving that
\begin{eqnarray}\label{finalest}
\H^n\left(\cup_j\Sigma_j\cap (\mathbb R^n\setminus \B)\right)\le \gamma/2\,.
\end{eqnarray}
To continue the proof we will need the following lemma.
\begin{lemma}\label{lemmaep}
Let $\ep>0$ be given.  Let $\Lambda_1$ be  as above and define
$A$ as in \eqref{const3}.  Then there exist  $\Lambda_2$, and $\Lambda_3$ as above,
and an integer $N_0=N_0(\epsilon,\Lambda_2)\geq 1$, such that if we let  $\Lambda_0=\Lambda_2^{\tilde N_0}$, for some $\tilde N_0\geq N_0$, and if we restrict $h$ to satisfy  $h\geq 2A\Lambda_0\Lambda_1\Lambda_3$, then the following is true. Let $j\geq 0$ and $I\subset I^0$ be a dyadic cube of length
$\ell(I)=A^{-j}$. Then  the number of dyadic cubes $J$ of length $\ell(J)=\Lambda_0^{-1}A^{-j}$ that are contained in $I$ and satisfy $J\cap (\Sigma_j\cap (\mathbb R^n\setminus \B))\neq\emptyset$,
is less than $\ep \Lambda_0^{n+1}$.
\end{lemma}
\begin{proof}[Remark on the proof] This is Lemma 2.2 in \cite{NS} and 
its proof does not rely directly on any
two cube condition. On the other hand, \cite[Lemma 2.2]{NS}
is deduced as a consequence of Lemma 2.5 in \cite{NS}, whose proof does
make nominal use of
the synchronized two cube condition.
However,  a careful examination of the argument reveals that
the weak synchronized two cube condition is sufficient. We omit further details.
\end{proof}

Let $\epsilon>0$ be a degree of freedom to be fixed in \eqref{efix} below.
To proceed with the proof of \eqref{finalest}, given $j\geq 0$  we dyadically subdivide $I^0$ into
(non-overlapping) dyadic  cubes $\{I_{j,l}\}_l$ of length $\ell(I_{j,l})=A^{-j}$.
Note that there are $A^{(n+1)j}$ such cubes, since $I^0$ is a unit cube.
We then subdivide
 each cube $I_{j,l}$ further and for
 $\Lambda_0$  as in Lemma \ref{lemmaep}, we let
 $\{J_{j,l,k}\}_{k=1}^{k_l}$
 denote the so constructed set of dyadic cube of length $\ell(J_{j,l,k})=\Lambda_0^{-1}A^{-j}$, satisfying $J_{j,l,k}\subseteq I_{j,l}$ and $J_{j,l,k}\cap\bigl (\Sigma_j\cap (\mathbb R^n\setminus \B)\bigr )\neq\emptyset$.
 By Lemma \ref{lemmaep} we have, for each $I_{j,l}$, that the cardinality $k_l$
 of the collection $\{J_{j,l,k}\}_k$ is at most $\ep \Lambda_0^{n+1}$.
We then have
\begin{align}\label{finalest+}
\H^n\left(\cup_j\Sigma_j\cap (\mathbb R^n\setminus \B)\right) &\leq \,\sum_j\sum_l \,\H^n\left(I_{j,l}\cap\bigl (\Sigma_j\cap (\mathbb R^n\setminus \B)\bigr)\right)\notag\\
&\leq\,  \sum_j\sum_l\sum_{k=1}^{k_l} \, \H^n\left(J_{j,l,k}\right) \, .
\end{align}
Hence, to prove \eqref{finalest}, it suffices to  show that
\begin{eqnarray}\label{finalest++}
\sum_j\sum_l\sum_{k=1}^{k_l} \,\H^n\left(J_{j,l,k}\right)\leq \gamma/2.
\end{eqnarray}
 To prove \eqref{finalest++} we will associate, to each $J_{j,l,k}$, a
 surface $S(J_{j,l,k})$, and we intend to estimate the measure of $\left|J_{j,l,k}\right|$ in terms  of the measures of the sets $\{S(J_{j,l,k})\}$. The surfaces will not be uniquely defined but as we will see we will make the construction so that
 $S(J_{j,l,k})\cap S(J_{j',l',k'})=\emptyset$ whenever $(j,l,k)\neq (j',l',k')$,
 thus enabling efficient summation.

 To proceed with the construction of the surface $S(J_{j,l,k})$, consider $J:=J_{j,l,k}$ and choose any ${\bf X}\in \G$ and ${\bf Y}\in S$ such that $\pi({\bf X})\in J$, ${\bf Y}\in {\bf X}+\Gamma$ and $A^{-j}\le y_n-x_n<A^{-j+1}$. Applying the weak time-synchronized corkscrew condition, at $\mathbf{Y}$ and at scale $\Lambda_1^{-1}A^{-j}$, we see that there exists a cube  $Q\subset\mathbb R^{n+1}$  of length
$$\ell(Q)= R_j:=M^{-1}\Lambda_1^{-1}A^{-j}\,,$$
with center $\bf U$ and contained in  $Q_{\Lambda_1^{-1}A^{-j}}(\mathbf{Y})$,
and such that $Q$ belongs  to a component of $\mathbb R^{n+1}\setminus\Sigma$ different from $\U$. We recall that $\U$ is the component that contains $I^{\bar M}$. However, in contrast to \cite{NS} the $t$-coordinates of ${\bf U}$ and ${\bf Y}$ {\bf  do not necessarily coincide}. This turns out to be harmless.  Given $J$ we let
\begin{equation}\label{eq5.21}
\hat J=\hat J(J):=I_{\Lambda_1^{-1}R_j}(\pi({\bf U})).
\end{equation}
 We also introduce
$$\S:= (\Sigma\cap Q_{2M}(0,0))\cup (I^0\times\{x_n=-A\})$$
and we recall that $A\geq 2M$. Given $J=J_{j,l,k}$ we define $S(J)$ to be the set of all ${\bf V}\in\S$ such that $\pi({\bf V})\in \hat J=\hat J(J)=I_{\Lambda_1^{-1}R_j}(\pi({\bf U}))$, with
$v_n<u_n-R_j$, where ${\bf U}=(u,u_n,\tau)$, and such that the open line segment joining ${\bf V}$ to $\pi({\bf V})+(0,u_n-R_j,0)$ does not
meet $\Sigma$. By construction, since  $\pi(\S)\supset \hat J$,
we have
\begin{equation}\label{eq5.22}
 \pi(S(J))=\hat J\,,\quad \text{and } \hat J\subset 2J\,,
 \end{equation}
where the latter holds since we have chosen $h$ very large.

To proceed,  let $K_j$ be the number of cubes $\{I_{j,l}\}$ that contain at least one of the $J_{j,l,k}$. Then, given $\epsilon$, $\tilde N_0\geq N_0$ and $h$ as stated in Lemma \ref{disjointlemma} below, and using
Lemma \ref{lemmaep},
\begin{align}\label{finalest+++}
\sum_l\sum_{k=1}^{k_l}\H^n\left(J_{j,l,k}\right)\leq K_j\ep \Lambda_0^{n+1}(2\Lambda_0^{-1}A^{-j})^{n+1}=K_j\ep 2^{n+1}A^{-j(n+1)}.
\end{align}
Fix $j,l$ and assume that the collection $\{J_{j,l,k}\}_k$ is non-empty.
Since  $J_{j,l,k}\subset I_{j,l}$, it follows that for all $k\in\{1,\ldots,k_l\}$, one
has $\pi(S(J_{j,l,k}))\subset 2I_{j,l}$,
and therefore by \eqref{eq5.21} and \eqref{eq5.22},
\begin{align}\label{finalest++++}
\H^n(\pi(\cup_{k=1}^{k_l}S(J_{j,l,k}))\cap 2I_{j,l})\geq (2\Lambda_1^{-1}R_j)^{n+1}.
\end{align}
Hence, summing the inequality in \eqref{finalest++++} over $l$
\begin{align}\label{finalest+++++}
\sum_l\H^n(\pi(\cup_{k=1}^{k_l}S(J_{j,l,k}))\cap 2I_{j,l})&\geq  K_j(2\Lambda_1^{-1}R_j)^{n+1}\notag\\
&=K_j(M^{-1}\Lambda_1^{-2})^{n+1}2^{n+1}A^{-j(n+1)}.
\end{align}
Combining \eqref{finalest+++} and \eqref{finalest+++++}, and using that for each given $j$,
the fattened cubes
 $\{2I_{j,l}\}_l$ have bounded overlaps,
we see that
\begin{align}\label{lemmaj}
\sum_l\sum_{k=1}^{k_l}\H^n\left(J_{j,l,k}\right)\leq C\varepsilon
(M\Lambda_1^{2})^{n+1}\, \H^n\left(\pi(\cup_l \cup_{k=1}^{k_l}S(J_{j,l,k}))\right)
\end{align}
 for all $j\ge 0$,  where the constant $C=C(n)$. 
 Hence, summing in $j$ we have
 \begin{align}\label{lemmaj+}
\sum_j\sum_l\sum_{k=1}^{k_l}\H^n\left(J_{j,l,k}\right)\leq C\varepsilon (M\Lambda_1^{2})^{n+1}\sum_j
\H^n\left(\pi(\cup_l\cup_{k=1}^{k_l}S(J_{j,l,k}))\right).
\end{align}
To complete the proof we will need the following lemma, Lemma \ref{disjointlemma}.

\begin{lemma}\label{disjointlemma} Let $\ep>0$ be given. Let   $\Lambda_2$,  $\Lambda_3$, $N_0$, be
as in the statement of Lemma \ref{lemmaep}. Then there
exists an integer $\tilde N_0\geq N_0$, depending only on $n,M$, $\Lambda_2$,  $\Lambda_3$, such that if we let $\Lambda_0=\Lambda_2^{\tilde N_0}$,
$\Lambda_1=2^{\tilde N_0}$,  define
$A$ as in \eqref{const3}, and if we restrict $h$ to satisfy $h\geq 2A\Lambda_0\Lambda_1\Lambda_3$, then
$$S(J_{j,l,k})\cap S(J_{j',l',k'})=\emptyset \,, \, \,\forall\,  l,k,l',k' \,, \mbox{ whenever $j\neq j'$.}$$
\end{lemma}
\begin{proof} This is Lemma 2.3 in \cite{NS}, 
whose proof relies in turn on Lemma 2.4 in \cite{NS}.
Neither of the proofs of these two Lemmata 
relies on a two cube condition, 
 hence Lemma \ref{disjointlemma} 
 generalizes immediately to our setting.
\end{proof}

We can now use  Lemma \ref{disjointlemma} to complete the proof of \eqref{finalest++} and hence the proof of Proposition \ref{mainprop}. Recall that  by definition, $$S(J_{j,l,k})\subset \S= (\Sigma\cap Q_{2M}(0,0))\cup (I^0\times\{x_n=-A\}).$$
Hence, using that $\H^n$ and $\Hpn$ are the same on a hyperplane parallel to the $t$-axis,
and that (parabolic)
Hausdorff measure does not increase under a
projection (see Remark \ref{r-measures} (ii) and (iii)), and then
Lemma \ref{disjointlemma}, we deduce that
\begin{multline}\label{aa1}
\sum_j \, \H^n(\pi(\cup_l \cup_{k=1}^{k_l}S(J_{j,l,k}))) \,
\leq\,
\sum_j \, \Hpn(\cup_l \cup_{k=1}^{k_l}S(J_{j,l,k})) \\
\le \,\Hpn\big(\Sigma\cap Q_{2M}(0,0) \big)+\H^n\left(I^0\times\{x_n=-A\}\right)\,
\leq \, C,
\end{multline}
since $\Sigma$ is ADR.
Together \eqref{lemmaj+} and \eqref{aa1} imply the bound
\begin{eqnarray}\label{finalest+++a}
\sum_j\sum_l\sum_{k=1}^{k_l}\H^n\left(J_{j,l,k}\right)\leq C\varepsilon (M\Lambda_1^{2})^{n+1},
\end{eqnarray}
where $C=C(n)$, $1\leq C<\infty$. Let now $\varepsilon$ be defined through the relation
\begin{eqnarray}\label{efix}
C\varepsilon (M\Lambda_1^{2})^{n+1}=\gamma/2.
\end{eqnarray}
Then $\ep=\ep(n,M,\Lambda_1,\gamma)=\ep(n,M,\gamma)=\ep(n,M)$ and we see that Lemma \ref{mainprop} holds with $h=2A\Lambda_0\Lambda_1\Lambda_3$ and, by construction, $h=h(n,M)$. In particular, the proof of Proposition \ref{mainprop} is now complete.

\section{The proof of  Theorem \ref{th2} and Corollary \ref{cor2}}\label{sec5}
In this section we give the proof of Theorem \ref{th2} when $\diam\Sigma=\infty$, $T_0=-\infty$ and $T_1=\infty$,
followed by a sketch of the refinements to this argument needed to prove Corollary \ref{cor2}.
The proof will be a combination of ideas in \cite{HLN1} and \cite{DS}.

In the following $C$ will denote a positive constant satisfying $1\leq C<\infty$. We write  $c_1\lesssim c_2$ if $c_1/c_2$ is bounded from above by a positive constant depending at most on $n$, $M$ and $\gamma_1$ if not otherwise stated. We write $c_1\sim c_2$ if $c_1\lesssim c_2$ and $c_2\lesssim c_1$.

\begin{proof}[Proof of  Theorem \ref{th2}] Let  $\Sigma$ be a closed subset of $\R^ {n+1}$ which is parabolic ADR with constant $M$.
Assume that $\Sigma$ is
parabolic UR with constants $(M,\|\nu\|)$, and that
and that $\Sigma$ satisfies a 2-sided corkscrew condition as in Definition \ref{corkscrews}.  Then by
Theorem \ref{urtscorkscrews},
$\Sigma$ satisfies the weak synchronized two cube condition in the sense of Definition \ref{tsynk} with $\gamma_1\in (0,1)$.  If necessary, we shrink $\gamma_1$ slightly so that any rotation $\varrho(Q)$ of a corkscrew cube $Q$
does not intersect $\Sigma$, and in fact retains the Whitney property that
$\diam\left(\varrho(Q)\right)\approx\diam(Q) \approx \dist\left(\varrho(Q),\Sigma\right)$, with uniform implicit constants.

 To start the proof of Theorem \ref{th2}, let $(X,t)\in \Sigma$ and $R>0$. By Theorem \ref{th1} there exists, after possibly a rotation in the spatial variables, a coordinate system and  Lip(1,1/2) function ${\psi}^\ast$
  with constant $b^* = b^*(n,M)$ such that if we let $\pi$ denote the orthogonal projection onto the plane $\{(y,y_n,s)\in\mathbb R^{n-1}\times\mathbb R\times\mathbb R:y_n=0\}$, then
\begin{align}\label{aa1+}
\sigma(F)\geq \mathcal{H}^n(\pi(F)) \geq \epsilon R^{n+1}\mbox{ where } F := \Sigma_{{\psi}^\ast}\cap \Delta(X,t,R)
\end{align}
and
\begin{align*}
\Sigma_{{\psi}^\ast} := \{(y,y_n,s) \in \mathbb{R}^{n-1} \times \mathbb{R} \times \mathbb{R}: y_n = {\psi}^\ast(y,s)\}.
\end{align*}

To prove Theorem \ref{th2} we need to invoke  the Carleson measure condition used in the very definition of parabolic uniform rectifiability.  Let
\[  f ( Z, \tau ) = \int_0^{100R}  \ga ( Z, \tau, r ) \,
r^{ - 1 } \, \d r, \, ( Z, \tau ) \in \Sigma.  \]
Then, using \eqref{1.9} we see that
\[ \iint_{ \Delta( X, t,100R)} f ( Z, \tau )\d\sigma
( Z, \tau ) \, \leq \, \| \nu \| \, (100 R)^{ n + 1 }.\]
  Using this and weak estimates we see that if  $ A=1000\epsilon^{-1}$, then
\begin{eqnarray} \si \left( \{  ( Z, \tau  ) \in  \Delta(X,t,100R) :  f
( Z, \tau) \geq \, A^{ n + 1} \, \| \nu \|  \} \right)&\leq& ( 100 R/A)^{ n + 1 } \notag\\
&\leq& (\epsilon R/10)^{ n + 1}.\end{eqnarray} Using this inequality, \eqref{aa1+} and the fact that Hausdorff measure does not
increase under a projection,  we deduce
the existence of a closed set $ F_1 = F_1 ( A ) $ with $ F_1
\subset F,  $ such that
\begin{eqnarray}\label{2.13}  f ( Z, \tau  ) \leq \, A^{ n + 1} \,
\| \nu \|, \qquad (Z, \tau ) \in F_1,  \end{eqnarray} and
\begin{eqnarray}\label{2.14}  \H^n( \pi( F_1)) \geq \frac {\epsilon}2 R^{n + 1 }. \end{eqnarray}

We construct the approximating graph by extending ${\psi}^\ast$ off $\pi( F_1)$. To do this we again identify $\mathbb R^{n-1}\times \{0\}\times\mathbb R$ with $ \mathbb R^{n}$, and  put
\[ I_r ( z, \tau  ) = \{ ( y, s ) \in \mathbb R^{n} :
| y_i   - z_i   | < r, \, i = 1,\dots, n - 1, \, \,   | s - \tau |
< r^2\}, \]
whenever $( z, \tau ) \in \mathbb R^{n}$, $r > 0$. Let  $ \{
\bar I_i = \overline {I_{r_i} ( \hat x_i, \hat t_i )} \} $ be a Whitney
decomposition of $ \mathbb R^{n} \sem  \pi( F_1) $ into ($n$-dimensional parabolic) cubes, such that
 $ I_i \cap I_j = \emptyset, i \not = j, $ and
\begin{eqnarray}\label{2.15} 10^{ - 10n }  d ( I_i, \pi( F_1) )  \leq  r_i   \leq 10^{ -
8n } d ( I_i,
\pi( F_1)).  \end{eqnarray}
Let $ \{ v_i \} $ be a partition of unity adapted to $ \{ I_i \}$, i.e.,
\begin{eqnarray}\label{2.16}
(a)&&\sum v_i
\equiv 1  \mbox{ on } \mathbb R^{n} \sem \pi( F_1),\notag\\
(b)&& v_i  \equiv 1  \mbox{ on }  I_i
 \mbox{ and }   v_i \equiv 0  \mbox{ in } \mathbb R^{n} \sem \overline{ I_{2
r_i } ( \hat x_i, \hat t_i )} \mbox{ for all } i,\notag\\
(c)&&v_i  \mbox{ is infinitely differentiable on $\mathbb R^{n}$ with}\notag\\
&& r_i^{ - l} \, | \frac{\partial^l}{\partial x^l }  v_i |
\, + \,    r_i^{ - 2 l}  | \frac{\partial^l}{\partial t^l }  v_i |  \,  \leq
\, c ( l, n ) \mbox{ for } l = 1, 2, \dots.
\end{eqnarray}
In ($c$), $
\frac{\ar^l}{\ar x^l } $ denotes an arbitrary partial derivative with respect to the space variable $ x $ and of order $l$. Next,
for each $ i $ we fix $ ( x'_i, t'_i ) \in \pi( F_1) $ with
\begin{equation}\label{eq6.8}
 \rho_i := d ( ( x_i', t_i' ),  I_i ) = d ( \pi( F_1), I_i ) \approx r_i \approx \diam(I_i),
\end{equation}
where the last two equivalences are standard properties of Whitney cubes.
We set   $  \La \, = \{ i :
\bar I_i \cap \overline {I_{2R} ( x, t ) }\not = \emptyset \},$
where $(y,s)\cong(y,0,s)$ is the projection of $(Y,s)$ onto $ \mathbb R^n\cong \R^{n-1}\times \{0\}\times \R$.  We now let
\begin{eqnarray}\label{2.17}
\,\, \psi ( y, s )  \, = \, \left\{ \bea{l} \psi^\ast ( y, s ), ( y, s )
\in \pi( F_1),  \\ \\
 { \ds \sum_{i \in \La} } \, \big(   \psi^\ast  ( x_i', t_i' )  \,+ \, \mu b^\ast  \rho_i
\,
\big) \, v_i ( y, s ) , \, \, ( y, s ) \in \mathbb R^{n} \sem \pi( F_1)\,, \ea
\right. \end{eqnarray}
where $\mu$ is a non-negative constant which may be taken equal to 0, in the case of
Theorem \ref{th2}, and which will be chosen sufficiently large in the case of
Corollary \ref{cor2}.
Then, $ \psi \equiv 0 $ on $
\mathbb R^{n}  \sem Q_{4R} ( X, t ) $, and 
\begin{eqnarray}\label{aa1+a}
 \H^{n}(\pi(F_1)) \geq \frac \epsilon2 R^{n + 1 }\,,\quad F_1\subset \Sigma_{\psi}\cap\Delta(X,t,R),
 \end{eqnarray}
 where
  \begin{eqnarray*} 
 \Sigma_{\psi}:=\{(y,y_n,s)\in\mathbb R^{n-1}\times\mathbb R\times\mathbb R:y_n=\psi(y,s)\}.
 \end{eqnarray*}

  We intend to prove that the function $\psi$ is a regular parabolic Lip(1,1/2)  function with constants
 $b_1=b_1(n,M,\tilde M)$, $b_2=b_2(n,M,\tilde M)$.

Since $ \psi^\ast $ is a Lip(1,1/2) function with constant $b^*=b^*(n,M)$,
one can use  \eqref{2.15}-\eqref{2.17} and a standard Whitney extension argument (see \cite[Ch. VI]{St})
to conclude that \eqref{1.1} holds with $ b_1 $ replaced by $ C b^\ast$.  To verify this, the more
delicate case occurs  when  $ (y,s) $ is in the
closure of two cubes say $ I_i, I_j $ with $ i \in \La, j \not \in
\La. $  However this case follows easily from the fact  that $ |
\psi^\ast | \leq c b^\ast R $ and $ | \ar v_k/\ar y_l | (y,s) \leq c/R $
for $ 1\leq l \leq n - 1$, $k\in\{ i, j\}$. Hence it remains only to prove that
 \begin{eqnarray} \label{1.7aa}
 \|D_{1/2}^t\psi\|_*\leq b_2\mbox{ for some }b_2=b_2(n,M,\tilde M).
\end{eqnarray}
Let $\beta_\psi$, $\nu_\psi$, be as in the
statement Definition \ref{def1.UR-} but with $\Sigma$ replaces by
$\Sigma_{\psi}$ as the underlying closed set. To prove \eqref{1.7aa} the key step is to prove that
\begin{eqnarray} \label{1.7aa+}\| \nu_\psi\| \lesssim ( 1+ \| \nu \| ).
\end{eqnarray}
Once \eqref{1.7aa+} is established, one can repeat the proof in \cite[pp 368-373]{HLN1} to
conclude that \eqref{1.7aa} holds with $b_2\approx  1+ \| \nu \|$,
thus completing the proof of Theorem \ref{th2}.

It therefore remains to give the proof of \eqref{1.7aa+}.
To start, we make an elementary observation of a
geometric nature. Indeed, we
first note that 
\begin{equation}\label{2.23} d (  Y, s, \Sigma ) \lesssim \, (1+b^*) 
\, d \big( y, s, \pi ( F_1 ) \big)\,,\quad \forall\,( Y, s ) \in \Sigma_{\psi}\cap \overline{Q_{100R} ( X, t )}\,.
\end{equation}
Indeed, this inequality is trival when $ ( Y, s) \in F_1$, so assume $ ( Y, s ) = ( y, s,\psi^\ast ( y, s) ) $
with $ ( y, s ) \in \bar I_i $ for some $i$.  Then $d ( y, s, \pi ( F_1 ) )\approx \rho_i
\approx d \big( (y,s),( x_i', t_i' )
\big) $, by 
\eqref{eq6.8}.
Consequently, since $\psi^\ast$ is Lip (1,1/2) with constant $b^*$,
\begin{equation}\label{2.23a}
d (   Y, s   , \Sigma) \leq d\big(( y, s,\psi^\ast ( y, s) ), ( x_{i}', t_{i}',\psi^\ast ( x_{i}', t_{i}' )\big)
\lesssim (1+b^*)\rho_i.
\end{equation}
This proves \eqref{2.23}.

In the following $K\gg 1 $ is a degree of freedom.
Given $(Z,\tau,r)\in \Sigma\times (0,\infty)$ we let $P_{(Z,\tau,r)}$ be a time-independent plane which realizes
$\beta(Z,\tau,Kr)$.

Consider
\begin{align}\label{case1}
\mbox{$(Z,\tau)\in F_1$ and $r>0$ such that $\overline{Q_r(Z,\tau)}\subset Q_{80R}(X,t)$}.
\end{align} Given $i\in\Lambda$, let $(X_i',t_i')\in F_1$ be such that $\pi(X_i',t_i')=(x_i',t_i')$ where
$ ( x'_i, t'_i ) \in \pi( F_1) $ realizes the distance from $I_i$ to $\pi( F_1)$.  Let $\mathcal{Q}_i$ be a
dyadic cube on $\Sigma$ (see Definition \ref{def.dyadiccube}) containing $(X_i',t_i')$  with
$\ell(\mathcal{Q}_i)\approx \rho_i$. Furthermore, let
	\begin{align*}
	&\Gamma_i= \{(y,\psi(y,s),s): (y,s) \in \overline{I_i}\}.
	\end{align*}
Then
\begin{equation}\label{eq6.17}
\sigma(\Gamma_i)\approx \rho_i^{n+1},
\end{equation}
(here we are using $\sigma$ to denote the surface measure both on $\Sigma$ and on $\Sigma_\psi$), and
\begin{equation}\label{eq6.18}
\rho_i=d( x'_i, t'_i ,I_i)=d\big(I_i,\pi(F_1)\big) \sim \ell(I_i) \sim \ell(\mathcal{Q}_i) \sim d(X_i',t_i',\Gamma_i)
\gtrsim d(\mathcal{Q}_i,\Gamma_i)\,,
\end{equation}
where in the next-to-last step we have used that $\Sigma_\psi$ is a Lip(1,1/2) graph. Using this notation we
see that
\begin{align}\label{aeq1}
\beta_\psi^2( Z,\tau, r )&\lesssim r^{-(n+1)}\iint_{\Sigma_\psi\cap Q_r( Z,\tau)}\biggl (\frac{d(Y,s,P_{( Z,\tau, r)})}r\biggr )^2\,\d\sigma(Y,s).
\end{align}
Introducing
\begin{align}\label{aeq2}
T( Z,\tau, r )&:= r^{-(n+1)}\iint_{F_1\cap Q_r( Z,\tau) }\biggl (\frac{d(Y,s,P_{( Z,\tau, r)})}r\biggr )^2\,\d\sigma(Y,s),\notag\\
T_i( Z,\tau, r )&:=r^{-(n+1)}\iint_{ \Gamma_i\cap Q_r( Z,\tau) }\biggl (\frac{d(Y,s,P_{( Z,\tau, r)})}r\biggr )^2\,\d\sigma(Y,s),
\end{align}
we can continue the estimate in \eqref{aeq1} and conclude that
\begin{align*}
\beta_\psi^2( Z,\tau, r )&\lesssim  T( Z,\tau, r )+\sum_{i\in I( Z,\tau, r )}T_i( Z,\tau, r ),
\end{align*}
where $ I( Z,\tau, r ):=\{i:\ Q_r( Z,\tau)\cap \Gamma_i \neq\emptyset\}$. By construction
\begin{align}\label{aeq3g}
 T( Z,\tau, r )\lesssim \beta^2( Z,\tau, Kr ).
\end{align}
To handle the sum over $i\in I( Z,\tau, r )$ we will combine arguments from \cite{DS} and \cite{HLN1}.

Let  $ i\in I( Z,\tau, r )$. Then
\begin{equation}\label{eq3a}
\rho_i \lesssim  r\mbox{ and } d(Z,\tau, \mathcal{Q}_i)\lesssim r\,.
\end{equation}
Choose a $(Z_i,\tau_i)\in \overline{\mathcal{Q}}_i$ which minimizes the distance from $\overline{\mathcal{Q}}_i$ to $P_{( Z,\tau, r)}$,
i.e.
\begin{equation}\label{eq3a+} h_i:=  \inf_{(Y,s)\in \mathcal{Q}_i}d(Y,s,P_{( Z,\tau, r)})= d(Z_i,\tau_i,P_{( Z,\tau, r)}).
\end{equation}
For $(Z_i,\tau_i)\in \overline{\mathcal{Q}}_i$ fixed as above, choose $\mathbf{Z}_{( Z,\tau, r)}\in P_{( Z,\tau, r)}$  so that
\begin{equation}\label{eq3a++}
h_i= d(Z_i,\tau_i,P_{( Z,\tau, r)})=d(Z_i,\tau_i,\mathbf{Z}_{( Z,\tau, r)}).
\end{equation}
Using this notation  and the triangle inequality, we write
\begin{align}\label{aeq2a}
T_i( Z,\tau, r )\lesssim \tilde T_i( Z,\tau, r )+\hat T_i( Z,\tau, r ),
\end{align}
where
\begin{align*}
\tilde T_i( Z,\tau, r )&:=r^{-(n+1)}\iint_{ \Gamma_i\cap Q_r( Z,\tau) }\biggl (\frac{d(Y,s,Z_i,\tau_i)}r\biggr )^2\,\d\sigma(Y,s),\notag\\
\hat T_i( Z,\tau, r )&:=r^{-(n+1)}\iint_{ \Gamma_i\cap Q_r( Z,\tau) }\biggl (\frac{d(Z_i,\tau_i,\mathbf{Z}_{( Z,\tau, r)})}r\biggr )^2\,\d\sigma(Y,s).
\end{align*}
We then have
\begin{align}\label{aeq2b}
\tilde T_i( Z,\tau, r )\lesssim (\rho_i/r)^{n+3}\,,\qquad \hat T_i( Z,\tau, r )\lesssim (\rho_i/r)^{n+1}(h_i/r)^{2},
\end{align}
where we have used \eqref{2.23} in the first estimate. Combining \eqref{aeq3g} and \eqref{aeq2b} we can conclude that if
$(Z,\tau ,r)$ is as in  \eqref{case1}, then
\begin{equation}\label{aeq3}
\beta_\psi^2( Z,\tau, r )\lesssim  \beta^2( Z,\tau, Kr )\,+\sum_{i\in I( Z,\tau, r )}\left(\frac{\rho_i}r \right)^{n+3}
\,+\sum_{i\in I( Z,\tau, r )} \left(\frac{\rho_i}r \right)^{n+1}\left(\frac{h_i}r \right)^{2}.
\end{equation}

We first treat 
the last term in \eqref{aeq3}, following the argument in \cite[pp 86-87]{DS}. Given $i\in I( Z,\tau, r )$ we set
$J(i):= \{j:\, \mathcal{Q}_j \subset \mathcal{Q}_i\}$,
and define
$$\mathcal{N}_i(Y,s):=    \sum_{j\in J(i)} 1_{\mathcal{Q}_j}(Y,s)$$ for $(Y,s)\in\Sigma$. Then, as in \cite{DS} we have
\begin{equation}\label{eq4uu}
\bariint_{\mathcal{Q}_i} \mathcal{N}_i\,\d\sigma \lesssim 1,
\end{equation}
and
\begin{equation}\label{eq3uu}
\sum_i \mathcal{N}_i(Y,s)^{-2} 1_{\mathcal{Q}_i}(Y,s) \lesssim 1\,.
\end{equation}
We sketch the proof of the latter estimate, as follows.  If $\mathcal{N}_i(Y,s)=\infty$, then trivially
$\mathcal{N}_i(Y,s)^{-2} =0$.  Otherwise, if $\mathcal{N}_i(Y,s)<\infty$, then there are only finitely many terms in the sum defining $\mathcal{N}_i(Y,s)$.
Note also that for each $k$, there is at most one $\mathcal{Q}_j\in\mathbb{D}_k$ such that $(Y,s)\in \mathcal{Q}_j$.  Thus, $\mathcal{N}_i(Y,s)$
equals the number of dyadic generations $k$ such that there is a cube $\mathcal{Q}_j\in \mathbb{D}_k$, with
$(Y,s)\in \mathcal{Q}_j\subset \mathcal{Q}_i$.  For the smallest $\mathcal{Q}_i$ containing $(Y,s)$, we have $\mathcal{N}_i(Y,s)=1$, for the next smallest
$\mathcal{N}_i(Y,s)=2$, etc., so that the sum in \eqref{eq3uu} is controlled by $\sum_{k=1}^\infty k^{-2}$.

Following \cite{DS}, we write
\begin{align*}
\rho_i^{n+1}  h^2_i = \rho_i^{n+1}  h^2_i
\left(\bariint_{\mathcal{Q}_i}d\sigma \right)^3 &=\rho_i^{n+1}  h^2_i
\left(\bariint_{\mathcal{Q}_i}\mathcal{N}_i^{-2/3} \mathcal{N}_i^{2/3}\d\sigma \right)^3\,.
\end{align*}
By H\"older's inequality, \eqref{eq3a+}, and  \eqref{eq4uu}, we then
deduce that
\begin{align*}
\rho_i^{n+1}  h^2_i  &\lesssim \iint_{\mathcal{Q}_i} (d(Y,s,P_{( Z,\tau, r)}))^2\,\mathcal{N}_i(Y,s)^{-2} 1_{\mathcal{Q}_i}(Y,s) \,\d\sigma(Y,s)
\, \left(\bariint_{\mathcal{Q}_i} \mathcal{N}_i \,\d\sigma\right)^2\\
&\lesssim  \iint_{\mathcal{Q}_i} (d(Y,s,P_{( Z,\tau, r)}))^2\,\mathcal{N}_i(Y,s)^{-2} 1_{\mathcal{Q}_i}(Y,s) \,\d\sigma(Y,s).
\end{align*}
Hence,  summing over $i$, using \eqref{eq3a} and \eqref{eq3uu}, we obtain
\begin{multline}\label{aeq3ha}
\sum_{i\in I( Z,\tau, r )} (\rho_i/r)^{n+1}(h_i/r)^{2}\\
\lesssim r^{-n-3}\iint_{\Sigma\cap Q_{Kr}(Z,\tau)}(d(Y,s,P_{( Z,\tau, r)}))^2 \,\d\sigma(Y,s)
\,\approx\, \beta^2( Z,\tau,Kr),
\end{multline}
provided that $K$ is chosen large enough, depending on the implicit constants in \eqref{eq3a}. In particular,
\begin{align}\label{aeq3again}
\beta_\psi^2( Z,\tau, r )&\lesssim  \beta^2( Z,\tau, Kr )+\sum_{i\in I( Z,\tau, r )}(\rho_i/r)^{n+3}
\end{align}
for all $(Z,\tau)\in F_1$ and $r>0$ such that $\overline{Q_r(Z,\tau)}\subset Q_{80R}(X,t)$.

For given $ ( \hat Z, \hat \tau ) \in \Sigma_\psi $ and $\hat r>0$ such that
with $ Q_{\hat r} ( \hat Z, \hat \tau ) \subset  Q_{20R} ( X, t
), $ we integrate \eqref{aeq3again} over  $ F_1 \cap Q_{\hat r} ( \hat Z,
\hat \tau ). $ If $ F_1 \cap Q_{\hat r} ( \hat Z, \hat \tau) =
\emptyset$ the following inequality is trivially true. Using \eqref{aeq3again}
\begin{align}\label{aeq3again+}
\nu_\psi( F_1 \cap Q_{\hat r} ( \hat Z, \hat \tau ) \times (0, \hat r))
 &=\int_0^{\hat r} \iint_{ F_1 \cap Q_{\hat r} ( \hat Z,
\hat \tau  ) }  \, \beta_\psi^2( Z,\tau,r ) \,\d\sigma ( Z,\tau) \, r^{ - 1 } \, \d r\notag\\
&\lesssim\, \nu( F_1 \cap Q_{\hat r} ( \hat Z, \hat \tau ) \times (0, K\hat r))\notag\\
&+ \,\,\,\int_0^{\hat r} \iint_{ F_1 \cap Q_{\hat r} ( \hat Z,
\hat \tau  ) }  \sum_{i\in I( Z,\tau, r )}(\rho_i/r)^{n+3}\,\d\sigma ( Z,\tau) \, r^{ - 1 } \, \d r\notag\\
& =: I + II.
\end{align}
By our assumptions, $I\lesssim \| \nu \| \,\hat r^{n+1}$.
Note that $ r_i' ( Z,\tau ) := d ( Z,\tau ,\Ga_i ) + \rho_i \,\lesssim \,r $, by \eqref{eq6.18} and \eqref{eq3a}.
Thus, summing and interchanging the
order of integration, we see that
\begin{align}\label{aeq3again++}
II\,&\leq \, \iint_{ F_1 \cap Q_{\hat r} ( \hat Z,
\hat \tau  ) }  \sum_{i\in I( \hat Z,\hat \tau, C\hat r )} \biggl (\int_{cr'_i ( Z,\tau ) }^{\hat r}\, (\rho_i/r)^{n+3}r^{ - 1 } \, \d r\biggr ) \,\d\sigma ( Z,\tau)\notag \\
&\lesssim \sum_{i\in I( \hat Z,\hat \tau, C\hat r )}  \iint_{ F_1 \cap Q_{\hat r} ( \hat Z,
\hat \tau  ) }   \left(\frac{\rho_i}{r'_i( Z,\tau )}\right)^{n+3} \,\d\sigma ( Z,\tau) \notag\\
&\lesssim  \sum_{i\in I( \hat Z,\hat \tau, C\hat r )} \rho_i^{n+1}\,\,\lesssim \,\,\hat r^{n+1}.
\end{align}
Hence, combining  \eqref{aeq3again+} and \eqref{aeq3again++}, we have proved that
\begin{equation}\label{aeq3again+++}
 \nu_\psi( F_1 \cap Q_{\hat r} ( \hat Z, \hat \tau ) \times (0, \hat r))\,\lesssim\, ( 1+ \| \nu \| )\,\hat r^{n+1}
\end{equation}
for all $ ( \hat Z, \hat \tau ) \in \Sigma_\psi $ and $\hat r>0$ such that
 $ Q_{\hat r} ( \hat Z, \hat \tau ) \subset  Q_{20R} ( X, t)$.

Similarly, by repeating the argument between displays (2.30) and (2.32) in \cite{HLN1} we first  deduce that
\begin{align}\label{aeq3again++++}
 \nu_\psi( (\Sigma_\psi\setminus F_1 )\cap Q_{\hat r} ( \hat Z, \hat \tau ) \times (0, \hat r))\lesssim ( 1+ \| \nu \| )\hat r^{n+1},
\end{align}
and then, using also \eqref{aeq3again+++}, we can conclude that
\begin{align}\label{aeq3again+++++}
 \nu_\psi( \Sigma_\psi\cap Q_{\hat r} ( \hat Z, \hat \tau ) \times (0, \hat r))\lesssim ( 1+ \| \nu \| )\hat r^{n+1},
\end{align}
whenever $ ( \hat Z, \hat \tau ) \in \Sigma_\psi $ and $\hat r>0$ are such that
with $ Q_{\hat r} ( \hat Z, \hat \tau ) \subset  Q_{20R} ( X, t
)$. The other cases can be handled by the observations in display (2.33) in \cite{HLN1}. We omit further details and claim
that the proof of \eqref{1.7aa+}, and hence the proof of Theorem \ref{th2}, is complete.
\end{proof}

\begin{proof}[Proof of Corollary \ref{cor2}]
Let $(\hat X,\hat t)\in \Omega$.  We repeat the proof of Corollary \ref{cor1} to construct a Lip(1,1/2) function,
which we now call $\psi^*$, along with the local graph subdomain $\tilde \Om=: \tilde\Om_{\psi^*}\subset \Om$,
defined as in \eqref{eq.tom} but with $\psi^*$ in place of $\psi$, above the
planar cube $I_*=\pi(Q_*)$ (see \eqref{eq.istar}).  With $\psi^*$ in hand, and using
\eqref{eq.tbp}, we repeat
the proof of Theorem \ref{th2}, with $\Delta(X,t,R)$ replaced by
$\Delta_*:= Q_* \cap\Sigma$, and thus $R\approx d/(Nh)$ (we recall that $Q_*$ is centered on
$\Sigma$; see Remark \ref{r5.2})).
Specifically, we construct $\psi$ as in
\eqref{2.17}, now with $\mu=N^{1/2}$, where $N$ is the suitably large constant in the proof of
 Corollary \ref{cor1}, and of course with $b^*=h$.  By the proof of Theorem \ref{th2}, $\psi$ is a regular Lip(1,1,2)  (i.e., $RPLip$)
 graph, as desired.
We now define $\tilde \Om=\tilde \Om_\psi$ again as in \eqref{eq.tom}, this time with
 respect to $\psi$.
 To obtain the conclusion of Corollary \ref{cor2}, it remains only to verify that
 $\tilde \Om_\psi\subset \Om$, and that the corkscrew condition \eqref{eq.tcs} holds for
 $\tilde\Om= \tilde \Om_\psi$.
 The former is easy: by construction (see \eqref{2.17}),
 $\psi^* \leq \psi$, pointwise in $I_*$, provided that $N$ (hence also
  $\mu=N^{1/2}$) is chosen large enough.  Thus,  $\tilde \Om_\psi \subset \tilde\Om_{\psi^*}$,
  and we already know that  in turn, $ \tilde\Om_{\psi^*}\subset \Om$.

  Let us now verify  that  \eqref{eq.tcs} holds for
 $\tilde\Om= \tilde \Om_\psi$.
    To this end, observe first that in the proof of Theorem \ref{th2}, by construction
  $\psi$ has compact support in a ball of parabolic radius $CR$, and
  that the planar Whitney cubes $I_i$ have ``length"
  $r_i\approx \rho_i \lesssim R$. In the present setting, this means that $\rho_i \lesssim d/(Nh)$
  for all $i$.  Since we have chosen $\mu=N^{1/2}$, and since $b^*=h$, this means that
  by construction (see \eqref{2.17}), applying \eqref{eq.psiest} to $\psi^*$, we have
 \begin{multline*}
 \sup_{(y,s) \in 100I_* }\psi(y,s) \leq \sup_{(y,s) \in 100I_* }\psi^*(y,s) + CN^{1/2} h N^{-1} h^{-1} d \\
 \leq \left (\kappa - c(n) + CN^{-1} +CN^{-1/2}\right) d \leq  (\kappa - c(n)/2)) d\,,
\end{multline*}
for $N$ large enough,
and therefore with $c_1=c(n)/2$, we have
\begin{equation}\label{eq6.39}
 \hat{x}_n - \sup_{(y,s) \in 100I_* }\psi(y,s) \geq c_1 d 
 \,.
\end{equation}
Combining the latter estimate with the definition of
 $\tilde\Om= \tilde \Om_\psi$ (see  \eqref{eq.tom}), we find that
 \eqref{eq.tcs} holds for $\tilde \Om_\psi$, provided that $N$ is chosen large enough.
\end{proof}

\section{Two Counterexamples}\label{sec6}

In \cite{NS}, the authors prove that a parabolic ADR set satisfying a synchronized two cube condition contains big
pieces of Lip(1,1/2) graphs.  It is quite easy to see that any set satisfying
a synchronized two cube condition is time-symmetric ADR.
It turns out that this implication is not reversible.  In particular, in light of Theorem \ref{tftbtscorkscrews}, the weak
time-synchronized two cube condition is {\em strictly} weaker than its strong counterpart.
 In this section, we construct two examples of time-symmetric ADR sets satisfying a
 (two-sided) corkscrew condition, which do
  {\it not} satisfy a synchronized two cube condition.
Moreover, our examples are also parabolic UR.
Importantly, these examples show that Theorems \ref{th1} and \ref{th2}, and Corollaries \ref{cor1} and  \ref{cor2}, are
strict improvements of the corresponding results in
\cite{NS}\footnote{As noted in the introduction, Theorem \ref{th2} and Corollary  \ref{cor2}
also improve the corresponding results in \cite{NS} in a further sense, namely that in the present work we
have removed the size constraint on the p-UR constants that was implicit in \cite{NS}.}.

The first example is rather simple:  let $\Omega$ be the open region between the two graphs $\Gamma^\pm
=\{(\psi^\pm(t),t\}\subset \R^2$, where
\[\psi^\pm(t):=  \pm |t|^{1/2} \pm 1\,,\quad t\in \R\,.\]
Clearly, $\Om$ is connected. Moreover,
it is easy to check that the boundary $\Sigma =\Gamma^+ \cup \Gamma^-$ is time symmetric ADR
(indeed, each of $\Gamma^\pm$ is a Lip(1,1/2) graph), and
satisfies the two sided corkscrew condition in the sense of Definition \ref{corkscrews+}
(i.e., with one point interior to $\Omega$ and one exterior).
On the other hand, $\Sigma$ fails to have {\em synchronized }
corkscrew points (one interior to $\Omega$ and one exterior)
in the sense of Definition \ref{tsynk+}, at $t=0$ (i.e., at the boundary points $(\pm1,0)$),
since the interior corkscrew points get pushed to the side at large scales.
Moreover, one may readily verify that each of the graphs $\Gamma^\pm$ is {\em regular} Lip(1,1/2),
and thus $\Sigma$ is p-UR, by checking the regularity criterion of \cite[Theorem 3.3]{Stz}, namely that each of
$\psi^\pm$ satisfies the Carleson measure condition
\[\sup_{a\in \R,\, h>0}\, \frac1h \int_{a-h}^{a+h} \int_{a-h}^{a+h} \frac{|\psi(t)-\psi(s)|^2}{|t-s|^2} \d s \d t \, \leq \, C\,.\]

We observe that the construction above does not provide
a counter-example to the time-synchronized 2-cube condition
in the weaker sense of Definition \ref{tsynk}, in which one merely insists upon the existence
of time-synchronized cubes in separate connected components of $\ree\setminus \Sigma$ (not necessarily interior to
one designated component).  Our next example and construction addresses this issue. The construction will be set in $\mathbb{R}^2$.

To start the construction in $\mathbb{R}^2$ we in this example will identify the horizontal axis as the time axis, and the vertical axis as the spatial axis. However, we will continue to denote points by $(X,t)$ where $X$ refer to the spatial coordinate and $t$ will refer to the time coordinate. 
Starting at $(0,0)$, we draw two line segment with slopes $\pm 2$, traveling distance $1/4$ on the time axis in the positive direction.
The endpoints of these two line segments will be $S_1 = \{\pm 1/2,1/4\}$.  Set $S_0 = \{(0,0)\}$, and $S_1 = \{(1/2,1/4),
(1/2,-1/4)\}$.  Also, we label the line segments constructed $\mathcal{G}_1$.  We will construct sets $\mathcal{G}_k$ and $S_k$. We will refer to the set $\mathcal{G}_k$ as the set of ``line segments of generation $k$," and we will refer to the set $S_k$ as the set ``branch points of generation $k$". We construct $\mathcal{G}_k$ and $S_k$ inductively as follows. We set $t_0=0$ and for $k \geq 1$ we set
\[t_k = \sum_{n=1}^k \frac{1}{4^n}. \]

Starting with a branch point $b$ of generation $1$, draw two line segments, each with initial vertex $b$, one having slope $4$, and the other slope $-4$, and each travelling $t$-distance
$1/4^2=1/16$.  Do this for each $b \in S_1$.  The resulting line segments define the set $\mathcal{G}_2$.  Additionally, the resulting branch points which define $S_2$  are
\[ S_2 = \{(3/4,t_2),(1/4,t_2),(-1/4,t_2),(-3/4,t_2)\}. \]

Now, we iterate this process (see the figure below).  From each branch point $b \in S_2$, we draw two line segments, one with slope $2^3$, and one with slope $-2^3$, and each with $t$-length $1/4^3$.  After doing this for all
$b \in S_2$, the resulting lines define $\mathcal{G}_3$, and the resulting branch points define $S_3$.  Proceeding inductively it is not hard to see that
\[S_n = \left\{\left(\pm \frac{2k+1}{2^n},t_n \right)\right\}_{k=0}^{2^{n-1}-1}, \]
for $ n\geq 2$. Note that, at generation $k$, the total distance travelled by the connected line segments of each previous generation is $t_n$ and $t_n\to 1/3$ as $n\to\infty$.

We set \[\Sigma_0 := \bigcup_{k=1}^\infty \bigcup_{l_\alpha \in \mathcal{G}_k} l_\alpha,\]
and we claim that
\[\displaystyle\overline{\Sigma_0} = \Sigma_0 \cup [ [-1,1]\times \{1/3\}].\]

\hspace{-5.5em}
\includegraphics{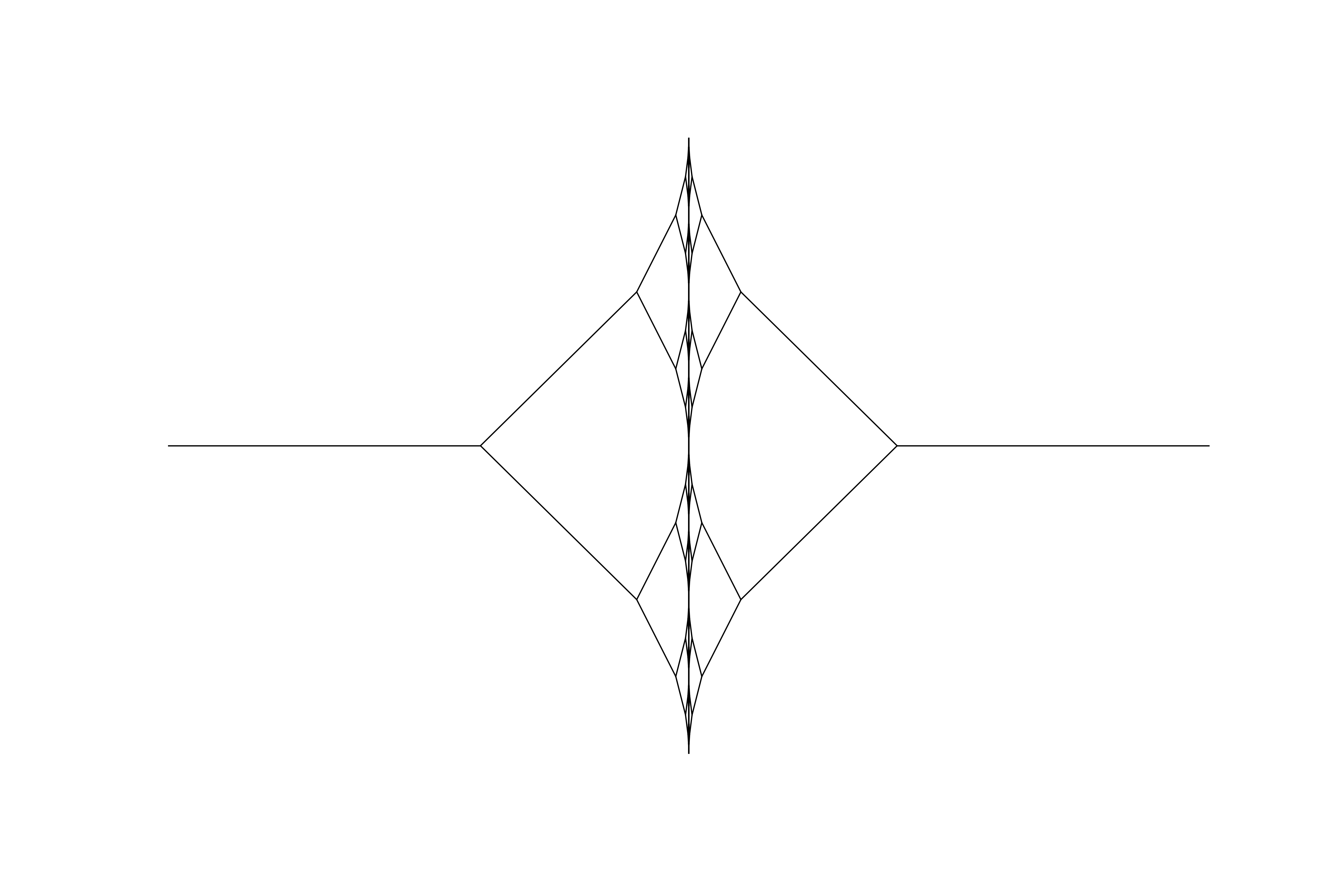}

To prove the claim, choose a point $a2^{-n}$ where $a \in \{\pm 1,\pm 2,...,\pm (2^n -1)\}$.  It is easily seen that $a2^{-n}$ is the spatial coordinate of some branch point in $\cup_{k=1}^n S_k$.  Let $a2^{-n}$ be the spatial coordinate of the branch point
$(a2^{-n},t_l) \in S_l$.  Then, choosing a ``child" branch point obtained by traveling down from $(a2^{-n},t_l)$ on the line segment with initial point $(a2^{-n},t_l)$ with slope $-2^{l+1}$ to the lower branch point of the next generation,
this ``child" branch point has spatial coordinate
$$a2^{-n} - 2^{l+1}4^{-l-1} = a2^{-n}-2^{-l-1}.$$
Next, suppose that we travel ``up" on every subsequent branch point.  Then the resulting spatial coordinates so obtained will converge to
$$a2^{-n} - 2^{-l-1} + \sum_{k=l+2}^\infty 2^{-l} = a2^{-n}-2^{-l-1}+2^{-l-1} = a2^{-n}.$$
Hence, $(a2^{-n},1/3)$ is a limit point of $\Sigma_0$.  By the arbitrary nature of the spatial coordinate $a2^{-n}$ we see that
$$\bigcup_{n=1} \bigcup_{a=1}^{2^n-1} (\pm a 2^{-n},1/3) $$
is in the closure of $\Sigma_0$.  Hence, it is easy to see that the claim follows.

Now, add the ray $(-\infty,0)$ to $\overline{\Sigma_0}$, and extend the resulting set by symmetry with respect to $t=1/3$.  We call the resulting set $\Sigma$.  The following is a computer-generated image of the set $\Sigma$ constructed (recall that the vertical axis represents $X$ and that the horizontal axis represents $t$).

We will prove that $\Sigma$ is parabolic UR and that $\Sigma$ satisfies a corkscrew condition.  First, let us focus
on showing that it is parabolic ADR.

First, we will show the ADR bounds on surface cubes centered at $(1/3) \times [-1,1]$.  Choose $(t,X) \in (1/3) \times [-1,1]$ and consider the surface cube $\Delta_R(t,x)$, $R>0$. 

Suppose first $R\geq 1$. We note that the measure of all of the line segments between $t=0$ and $t=1/3$ is
$$\sum_{n=1}^\infty 2^{n}4^{-n} =1.$$
Hence, the measure of all of the line segments between $t=1/3$ and $t=2/3$ is also $1$.  So, for $R\geq 1$,
$$\sigma(\Delta_R(X,t)) \leq 2 + 2(R-1)^2 \lesssim R^2,$$
where the factor $2(R-1)^2$ accounts for the possibility that the rays $(-\infty,0)$ and $(1/3,\infty)$ intersect the surface cube.  Now, suppose that $R\leq 1$, and that $R \approx 2^{-k}$ for some $k>0$.  We want to estimate the integers $m$ such that the backward face of
$Q_R(X,t)$ has $t$-coordinate $\approx t_m$, but
$$ \sum_{n=1}^m 4^{-n} \approx \frac{1}{3} - 4^{-k} \implies \frac{1}{3} - \frac{4^{-m}}{3} \approx  \frac{1}{3} - 4^{-k} \implies m \approx k.$$
If $\Delta_R(X,t)$ intersect segments of generation $k$, it will pick up approximately $2^{-k}$ of the total measure of the segments of that generation.  Hence
$$ \sigma(\Delta^{-}_R(X,t)) \approx 2^{-k}\sum_{n=k}^\infty 2^{-n} = 2^{-k}2^{1-k} \approx 2^{-2k} \approx R^2.$$
Because $\sigma(\Delta^{-}_R(X,t))=\sigma(\Delta^{+}_R(X,t))$ by symmetry, this establishes the upper and lower ADR bounds for cubes centered on the vertical face $ [-1,1]\times\{1/3\}$.

Now, suppose that $(X,t) \in \Sigma \cap \{0\leq t \leq 2/3\} \setminus \{t=1/3\}$.  Without loss of generality, we can assume that $0<t<1/3$.  Let $(X,t)$ lie on a line segment of generation $k$.  First,
we consider the case when $R \geq 2^{-k-1}$.  Then, $\Delta_R(X,t)$ is contained a surface cube of size $\approx R$ (but greater than $R$) centered on the vertical face $[-1,1]\times\{1/3\}$.
From this, we easily see that the upper ADR bound holds in this case.  The lower ADR bound holds trivially in both the forward and backward directions.
So, assume that $R < 2^{-k-1}$, so that the surface cube $\Delta_R(X,t)$ does not intersect the vertical face.  In fact, we can see that
the surface cube only intersects a uniformly bounded number of lines of generation $k-1$, $k$ and $k+1$.  So, it is easy to see that
$$\sigma(\Delta_R(X,t)) \approx R^2,\ \sigma(\Delta^+_R(X,t)) \approx R^2,\ \sigma(\Delta^-_R(X,t)) \approx R^2.$$
Now, the last case to consider is when $(X,t)$ lies in one of the rays $(0,\infty) \times \{0\}$, $(1/3,\infty) \times \{0\}$.  This case is easy to see.  We leave the details to the reader.

Now, we show that the set is parabolic UR.  First, we show Carleson measure estimates hold on the points of the vertical face $[-1,1]\times \{1/3\} $.   Choose a point $(X,t) \in [-1,1]\times\{1/3\}$, and $R>0$.  Choose $l$ to be the largest
 integer such that $R \leq 2^{-l}$.  We split 
\begin{align*}
&\int_0^R \int_{\Delta_R(X,t)} \beta^2(Y,s,r) \frac{\d\sigma(Y,s)\d r}{r}  \\
&\leq \sum_{k=l}^{\infty} \int_0^{2^{-k}} \int_{\Delta_R(X,t) \cap \mathcal{G}_k} \beta^2(Y,s,r) \frac{\d\sigma(Y,s)\d r}{r}\\
&+ \sum_{k=l}^{\infty} \int_{2^{-k}}^{2^{-l}} \int_{\Delta_R(X,t) \cap \mathcal{G}_k} \beta^2(Y,s,r)\frac{\d\sigma(Y,s)\d r}{r} \\
&=: I + II.
\end{align*}
Here, $\mathcal{G}_k$ refers to both the ``original" lines of generation $k$, and their reflections about $t=1/3$.  First, let us deal with term $I$.  We note that for $(s,Y) \in \mathcal{G}_k$, is it easy to see that
\begin{align*}
\beta^2(Y,s,r) \lesssim \frac{1}{r^4} \int_0^{r^2} |2^kt^2|^2 \d t = \frac{2^{2k}}{r^4} \int_0^{r^2} t^2 \d t \approx 2^{2k}r^2.
\end{align*}
Hence, as $\Delta_R(X,t)$ only intersects $\mathcal{G}_k$ for $k \geq l$,
\begin{align*}
I &\lesssim \sum_{k=l}^\infty \int_0^{2^{-k}} \int_{\Delta_R(X,t) \cap \mathcal{G}_k} 2^{2k}r^2 \frac{\d\sigma(Y,s)\d r}{r}\\
& \lesssim \sum_{k=l}^\infty \int_0^{2^{-k}} 2^{-l}2^{-k}2^{2k}r^2 \frac{\d\sigma(Y,s)\d r}{r} \approx 2^{-2l} \approx R^2.
\end{align*}
Now, we need to estimate term $I$.  For this, we simply note that the $\beta$ numbers are all uniformly bounded by a constant which depends only on ADR.  Hence
\begin{align*}
II &\leq \sum_{k=1}^\infty \int_{2^{-k}}^{2^{-l}} \int_{\Delta_R(X,t)\cap \mathcal{G}_k} \beta^2(Y,s,r) \frac{\d\sigma(Y,s)\d r}{r}\\
& \lesssim \sum_{k=l}^\infty \int_{2^{-k}}^{2^{-l}} \int_{\Delta_R(X,t)\cap \mathcal{G}_k} \frac{\d\sigma(Y,s)\d r}{r}\\
&\leq \sum_{k=l} 2^{-l}2^{-k} k \lesssim 2^{-2l} \approx R^2.
\end{align*}
So, we have appropriate Carleson measure bounds for surface cubes centered on the vertical face.  Now, we need to prove the same estimates for points in $\Sigma \cap \{0\leq t \leq 2/3\} \setminus ([-1,1]\times\{1/3\})$.
Again, just as in proving the ADR bounds, we can reduce this to proving the bound for points $(X,t) \in \Sigma$ with $0<t<1/3$.  Choose such a point $(X,t)$, and suppose that it lies on a line segment of generation $k$.
If we choose a scale $R \geq 2^{-k+1}$, then, again, there is a surface cube $\Delta_{CR}$ centered on the vertical face, containing $\Delta_R(X,t)$.  Therefore
$$\nu[\Delta_R(X,t)\times (0,R)] \leq \nu[\Delta_{CR} \times(0,CR)] \lesssim R^2.$$
On the other hand, if $r< 2^{-k}$, then the appropriate Carleson measure bound follows immediately from estimating a term like $I$ above.  Finally, all that is left is to prove the Carleson measure estimate for surface cubes which are not centered at point with
$t$-value between $0$ and $2/3$.  This case is easy, and we leave the details to the reader.

Now, we need to prove that $\Sigma$ satisfies a two-sided corkscrew condition.  Choose a point $(X,t)$ on the vertical face, and a scale $R$.  If $R \geq 1$, then $\Delta_R(X,t)$ will contain an portion of the ray $(0,\infty) \times \{0\}$ of $t$-length $\approx
R$.  It is easy to produce corkscrews in this case by considering points on the portion of the ray contained in the surface cube.  Now, suppose that $R \leq 1$.  then $R \approx 2^{-k}$ for some $k\geq 1$.  It is easy to see that $\Delta_R(X,t)$ will completely contain a line segment in $\mathcal{G}_l$ for some $l \approx k$.  At the midpoint of this line segment, it is also easy to see that we can produce corkscrews at scale $2^{-l}$, each of parabolic size $\approx 2^{l}$.  Now, consider the case where $(X,t)$ lies on a line segment in $\mathcal{G}_k$.  For $R<2^{-k+1}$, it is trivial to show that we can produce corkscrews of size $\approx R$ which lie within $Q_R(X,t)$.  Now, suppose that $R \geq 2^{-k+1}$.  Then there is a surface cube $\Delta_{R/2}$ centered on the vertical face contained in $\Delta_{R}(X,t)$.  By the work above, we can produce corkscrews relative to $\Delta_{R/2}$ of parabolic size $\approx R/2 \approx R$, which are clearly corkscrews relative to
$\Delta_R(X,t)$.  Finally, we need to consider when $(X,t)$ lies outside of $\{(X,t):0\leq t \leq 2/3\}$.  But this case is trivial.

Note that as $\Sigma$ contains a vertical face, it is impossible for $\Sigma$ to satisfy a synchronized two-cube condition.  However, $\Sigma$ is parabolic UR, so in fact it satisfies a weak synchronized two cube condition.



\end{document}